\documentclass[final,onefignum,onetabnum]{siamart190516}
\usepackage{caption}

\usepackage{lipsum}
\usepackage{amsfonts}
\usepackage{amsmath,amssymb}
\usepackage{graphicx}
\usepackage{epstopdf}
\usepackage{algorithmic}
\usepackage{cleveref}

\newsiamremark{remark}{Remark}
\newsiamremark{hypothesis}{Hypothesis}
\crefname{hypothesis}{Hypothesis}{Hypotheses}
\newsiamthm{claim}{Claim}

\headers{Iterative eigenvalue algorithm}{M. Kenmoe, R. Kriemann, M. Smerlak and A. Zadorin}

\title{A fast iterative algorithm for  near-diagonal eigenvalue problems\thanks{Submitted to the editors in Dec. 2020. \funding{Funding for this work was provided by the Alexander von Humboldt
Foundation in the framework of the Sofja Kovalevskaja Award endowed by
the German Federal Ministry of Education and Research.}}}

\author{Maseim Kenmoe\thanks{University of Dschang, Cameroon, and Max Planck Institute for the Mathematical Sciences, Leipzig, Germany}
\and Ronald Kriemann\thanks{Max Planck Institute for the Mathematical Sciences, Leipzig, Germany}
\and Matteo Smerlak\footnotemark[3]
\and Anton S. Zadorin\footnotemark[3]}

\usepackage{amsopn}
\DeclareMathOperator{\diag}{diag}

\newcommand{\D}{\partial}

\newcommand{\Z}{Z} 
\newcommand{\eigval}{\lambda} 
\newcommand{\Eigval}{\Lambda} 
\newcommand{\parameter}{\varepsilon} 
\newcommand{\evzero}{\mathring{\lambda}} 

\begin{document}

\maketitle

\begin{abstract}
We introduce a novel eigenvalue algorithm for near-diagonal matrices inspired from Rayleigh-Schr\"odinger perturbation theory and termed \textit{iterative perturbative theory} (IPT). Contrary to standard eigenvalue algorithms, which are either `direct' (to compute all eigenpairs) or `iterative' (to compute just a few), IPT computes any number of eigenpairs with the same basic iterative procedure. Thanks to this perfect parallelism, IPT proves more efficient than classical methods (LAPACK or CUSOLVER for the full-spectrum problem, preconditioned Davidson solvers for extremal eigenvalues). We give sufficient conditions for linear convergence and demonstrate performance on dense and sparse test matrices, including one from quantum chemistry. The code is available at \url{http://github.com/msmerlak/IterativePerturbationTheory.jl}.
\end{abstract}



\section{Introduction}

How can one compute the eigenvectors of a matrix $M$ that is already close to being diagonal, or equivalently when approximate eigenvectors are known? In physics and chemistry, a standard approach, Rayleigh-Schr\"odinger perturbation theory \cite{rayleigh1894, kato1995},  attempts to compute these eigenvectors as power series in the `perturbation' $\Delta = M - D$, where $D$ is diagonal and $\Delta$ is small. The purpose of this paper is to show that a variant of this method, based on fixed-point iteration rather than power series expansion, can be the basis for a new kind of preconditioned eigenvalue algorithm. We stress from the outset that, like Rayleigh-Schr\"odinger perturbation theory, this algorithm (termed Iterative Perturbation Theory, IPT for short) has a limited scope: if $M$ is not sufficiently close to being diagonal, IPT will not converge. But in those cases---common in physical \cite{rayleigh1894, pozar1998} or chemical \cite{ostlund1996} applications---where $M$ is in fact near-diagonal, IPT proves much more efficient than established methods, sometimes by an order of magnitude or more.

Eigenvalue algorithms are usually classified in `direct' and `iterative' methods \cite{demmel1997}. These correspond to very different approaches to the eigenvalue problem. Direct methods, such as the QR and divide-and-conquer algorithms, reduce the the matrix into Hessenberg (or tridiagonal) form before they can compute complete set of eigenvectors.\footnote{One exception is Jacobi's algorithm for symmetric matrices, which annihilates off-diagonal elements with two-dimensional rotations. Because the complexity of Jacobi's algorithm is larger than that of tridiagonal reduction, this algorithm is not used in practice, except perhaps for small problems on GPUs.} Unfortunately, this initial step is difficult to parallelize, breaks any sparsity patterns in $M$, and has unfavorable (cubic) complexity. For this reason, direct algorithms are only applicable to small matrices with size $n \lesssim 10^4$. 

Iterative methods, on the other hand, only aim to compute a few eigenpairs of $M$ through `matrix-free' applications of $M$ to vectors. Because they do not require reducing (or even explicitly forming) $M$, iterative methods can be applied to much larger problems than direct approaches, particularly when $M$ is sparse. Central to most of these methods is the Rayleigh-Ritz procedure: at each iteration, a small subspace is formed (e.g. by expansion from a smaller subspace), and the eigenvectors of the matrix projected to that subspace are computed. This is the basis for classical Krylov methods (Arnoldi, Lanczos), but also for preconditioned methods which use an approximation of $M^{-1}$ to speed up convergence. The latter include Davidson methods (Generalized Davidson \cite{Davidson_1975, Morgan_2009}, Jacobi-Davidson \cite{fokkema1998jacobi, sleijpen2000jacobi}), Locally Optimal Block Preconditioned Conjugate Gradient (LOBPCG) \cite{knyazev2017recent}, etc. 

Our algorithm has a different structure. For starters, IPT makes no distinction between direct and iterative procedures: it computes any desired number of eigenpairs in parallel, from one to all of them, using the same basic iterative procedure. Second, IPT does not involve any reduction of the matrix, relying instead on repeated products $M\times Z$, where $Z$ is an $n\times k$ matrix when $k$ eigenpairs are requested. If $k = 1$, this is a matrix-vector product, as in iterative methods; when $k = N$ (the full spectrum problem), this is a matrix-matrix product, with sub-cubic theoretical complexity $\mathcal{O}(n^{2.376})$. Third, IPT does not rely on the Rayleigh-Ritz procedure or external eigensolvers to solve a projected problem---it is completely self-contained. Fourth, existing methods normally distinguish between symmetric (or Hermitian) and non-symmetric problems. IPT does not assume or use symmetry and is equally efficient with non-symmetric problems. Finally, and most importantly in our view, perhaps, IPT reduces the eigenvalue equation into an elementary fixed-point problem. This means that the vast repertoire of methods developed in numerical for solving non-linear equations, including fixed-point acceleration methods, become directly relevant to the eigenvalue problem. We feel that this connection between non-linear analysis and numerical linear algebra is the most appealing aspect of IPT.

Our presentation starts with a reformulation of the eigenvalue equation as a fixed point equation for a quadratic map in complex projective space (\cref{fixed-points}). We then establish a sufficient condition for fixed point iteration to converge and illustrate its divergence for larger perturbations with a simple two-dimensional example (\cref{convergence}). An interesting aspect of IPT is its compabibility with fixed-point acceleration methods (\cref{acceleration}). Next, we consider the computational efficiency of our method on multi-CPU and GPU architectures using test matrices, either sampled at random or drawn from chemistry applications (\cref{performance}). We conclude with some possible directions for future work.

\section{Definitions and notations}\label{definitions}

Below we denote complex vectors and functions to vector spaces by bold face lowercase letters: $\mathbf v \in \mathbb C^n$. We denote coordinates of vector $\mathbf v$ in the standard basis of $\mathbb C^n$ either by $(\mathbf{v})_i$ or by $v_i$, using the same letter but in Roman. For matrices and functions to spaces of matrices we use upper case Roman letters. Superscripts in parentheses (e.g. $x^{(k)}$) denote the order of an approximation, while the superscripts in square brackets (e.g. $x^{[k]}$) enumerate terms in asymptotic series.

Everywhere in the text we understand by \emph{genericity} of an object in some set $\mathbb C^n$ with the standard topology the condition that that object belongs to a fixed, open, everywhere dense subset of $\mathbb C^n$. A phrase ``condition $A$ is false for a generic object $O \in S$'' should be read as ``condition $A$ corresponds to a closed nowhere dense subset of $S$''. For example, genericity of an $n\times n$ matrix $M$ is understood with $M$ viewed as an element of $\mathbb C^{n\times n}$. A phrase ``eigenvalues of a generic $n\times n$ matrix are pairwise different'' should be understood as ``the set of matrices with eigenvalues of higher multiplicity form a nowhere dense closed subset in $\mathbb C^{n\times n}$''. Genericity of a family of partitions (defined below) $D + \parameter \Delta$, where $D = \diag(\mathbf{v})$, $\mathbf{v} \in \mathbb C^n$, and $\Delta \in \mathbb C^{n\times n}$, $\parameter \in \mathbb C$, is understood as genericity of a point $(\mathbf v,\Delta,\parameter) \in \mathbb C^n \times \mathbb C^{n\times n} \times \mathbb C$, and so on.

We use some elementary notions from projective and differential geometry. Their definitions are given below.

\begin{figure}
    \centering
    \includegraphics[width=\textwidth]{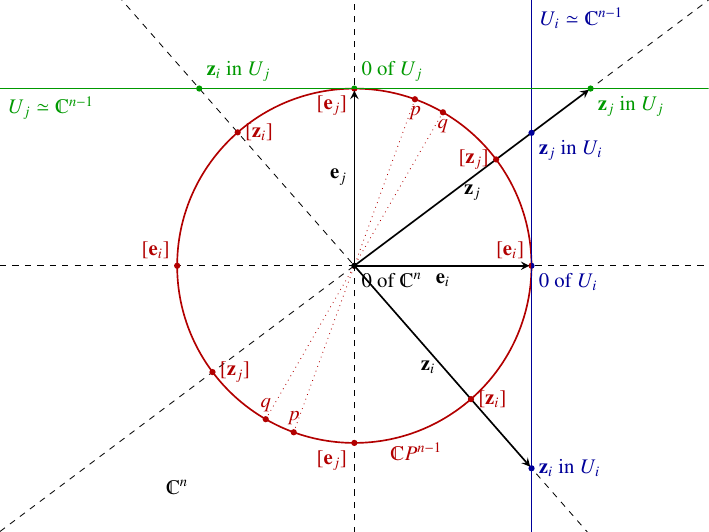}
    \caption{
        A schematic sketch of the geometrical relations between the original space $\mathbb{C}^n$, the projective space $\mathbb{C}P^{n-1}$, its affine charts, and points in these different spaces that correspond to certain eigenvectors. The projective space $\mathbb{C}P^{n-1}$ (in red) can be associated with the unit sphere in $\mathbb{C}^n$, in which each circle carved out in it by a complex line passing through the origin (a special case of a real plane) is glued to form a single point (here exemplified by two abstract points $p$ and $q$, the circles are degenerated to pairs of points in real settings). The standard basis $\{\mathbf{e}_k\}$ of $\mathbb{C}^n$ (eigenvectors of $D$) defines affine charts $U_k$ of $\mathbb{C}P^{n-1}$ (only two are depicted: for $k=i$ in blue and $k=j$ in green). $U_i$ and $U_j$ are shown as affine subspaces of $\mathbb{C}^n$. Note that $\mathbf{e}_j$ is at infinity of $U_i$ and $\mathbf{e}_i$ is at infinity of $U_j$. An actual eigenvector and all its nonzero rescalings correspond to a single point in $\mathbb{C}P^{n-1}$ and in each $U_k$. Two eigenvectors $\mathbf{z}_i$ and $\mathbf{z}_j$ are depicted with a rescaling choice such that they are equal to the $i$-th and $j$-th columns of matrix $\Z$.
    }
    \label{sketch}
\end{figure}

Informally, the $(n-1)$-dimensional projective space $\mathbb{C}P^{n-1}$ is the space of complex lines (1-dimensional complex subspaces) of the vector space $\mathbb{C}^n$ (the notion is generalizable to vectors spaces over fields different from $\mathbb{C}$). It is constructed in the following way. Consider an equivalence relation of \emph{nonzero} vectors $\mathbf{v},\mathbf{w}\in\mathbb{C}^n$ given by $\mathbf{v} \sim \mathbf{w}$ if $\exists \alpha \in \mathbb{C}$ such that $\mathbf{v} = \alpha \mathbf{w}$. In other words, all collinear vectors, with exception of $0$, are declared to be equivalent.

\begin{definition}
The \emph{projective space} $\mathbb{C}P^{n-1}$ is the factor-space $(\mathbb{C}^n\setminus\{0\})/{\sim}$ by the indicated equivalence relation with the associated factor-topology.
\end{definition}

We care about eigenvectors only up to normalization, therefore the projective space is a natural space that contains them as its points. As usual, we denote the equivalence class of a vector $\mathbf{z}\in\mathbb{C}^n$ (a point in $\mathbb{C}P^{n-1}$) as $[\mathbf{z}]$, and $\mathbf{z}$ is called a representative of that class. The coordinates of $\mathbb{C}^n$ can be used to (nonuniquely) parametrize points of $\mathbb{C}P^{n-1}$ in the following way.

\begin{definition}
A tuple $[z_1:z_2:\dots:z_n]$ is called \emph{homogeneous coordinates} of a point $p\in\mathbb{C}P^{n-1}$ if there is a vector $\mathbf{z}\in\mathbb{C}^n$ such that $p = [\mathbf{z}]$ and $\mathbf{z}$ has coordinates $(z_1,z_2,\ldots,z_n)$.
\end{definition}

Thus each point of the projective space has a set of homogeneous coordinates, but all of them are related by homogeneous rescaling: two homogeneous coordinates of the same point $[x_1:\dots:x_n]$ and $[y_1:\dots:y_n]$ are related by $x_i = \alpha y_i$ with the same $\alpha$ for each $i$.

$\mathbb{C}P^{n-1}$ has a structure of a complex $(n-1)$-dimensional manifold.

\begin{definition}
A \emph{complex (coordinate) chart} on a topological space $M$ is a homeomorphism $\phi\colon V\to O\in\mathbb{C}^n$ from some open set $V \subset M$ to an open set $O$ of $\mathbb{C}^n$ with the standard topology for some dimension $n$.
\end{definition}

It is convenient to denote a chart $\phi$ defined on $V\subset M$ with $(V,\phi)$ explicitly marking its domain.

\begin{definition}
Let there be two complex charts $(V,\phi)$ and $(W,\psi)$, $V,W \subset M$, and let their ranges have the same dimensionality $n$. The homeomorphism $\phi\circ\psi^{-1}$ between $\psi(V\cap W) \subset \mathbb{C}^n$ and $\phi(V\cap W) \subset \mathbb{C}^n$ is called a \emph{transition function} (between charts $\psi$ and $\phi$).
\end{definition}

\begin{definition}
A topological space $M^n$ is called a \emph{complex $n$-dimensional manifold} if 1) there is a set of charts $\{(U_i,\phi_i)\}_{i\in I}$ ($I$ is some index set) of $M^n$ such that $\bigcup_{i\in I} U_i = M^n$, 2) all the ranges of $\phi_i$ have the same dimensionality $n$ and 3) all the transition functions $\phi_i\circ\phi_j^{-1}$ are holomorphic functions between open subset of $\mathbb{C}^n$. The system $\{(U_i,\phi_i)\}_{i\in I}$ is called a \emph{holomorphic atlas} on $M^n$.
\end{definition}

An atlas allows to work with a manifold using usual methods of linear algebra and analysis in $\mathbb{C}^n$. The standard atlas of charts that gives $\mathbb{C}P^{n-1}$ a structure of an $(n-1)$-dimensional complex manifold is given by $n$ special charts. For each $i$, $1 \leq i \leq n$, define $U_i = \{[\mathbf{z}]:z_i \neq 0\}$, label coordinates of $\mathbb{C}^{n-1}$ by $\zeta_1$, \ldots, $\zeta_{i-1}$, $\zeta_{i+1}$, \ldots, $\zeta_n$ ($i$-th label is omitted) and set $\zeta_k = z_k/z_i$, which defines a map $\phi_i\colon U_i\to\mathbb{C}^{n-1}$. This defines a complex chart on $\mathbb{C}P^{n-1}$ with the whole of $\mathbb{C}^{n-1}$ as its image. It is easy to see that $\{U_i,\phi_i\}_{1\leq i\leq n}$ forms a holomorphic atlas of $\mathbb{C}P^{n-1}$. The transition functions from $\phi_i$ with its range coordinatized by $\zeta_k$ and $\phi_j$ with its range coordinatized by $\eta_k$ (with the same labeling omission rule for each chart as above) is given by $\eta_k = \zeta_k/\zeta_j$ for $k \neq i,j$ and $\eta_i = 1/\zeta_j$. They are simple rational maps and therefore are holomorphic.

In fact, the action of $\phi_i$ on $[\mathbf{z}]$ can be understood as setting $z_i = 1$, rescaling the rest of $z_k$ accordingly and taking the whole tuple $(z_1,\ldots,z_n)$ as coordinates of a point in the original space $\mathbb{C}^n$. The image of the chart $(U_i,\phi_i)$ becomes an $(n-1)$-dimensional affine subspace of the original space. In the following we will abuse notations by denoting the domain of the $i$-th chart, the chart itself, and the corresponding affine subspace with the same symbol $U_i$. We will refer to these special charts as \emph{affine charts}. Geometric relationships between $U_i$ and the action of the transition functions as projections between them is schematically shown in \cref{sketch} (other notations on this sketch will be introduced below).

Polynomials and their systems can be used to define subsets on linear complex spaces.

\begin{definition}
Let $\{P_k\}_k$ be a system of polynomials in $n$ variables $\{z_i\}_i$. The subset $V = \{\mathbf{z} \in \mathbb{C}^n: \forall\,k\,P_k(z_1,\dots,z_n) = 0\}$ is called an \emph{algebraic variety}.
\end{definition}

A polynomial $P$ such that $\forall \alpha \in \mathbb{C}$ $P(\alpha z_1,\dots,\alpha z_n) = \alpha^k P(z_1,\dots,z_n)$ is called a \emph{homogeneous polynomial}. It is obvious that if a point $\mathbf{z}$ is a root of a homogeneous polynomial, then every representative of its equivalence class $[\mathbf{z}]$ is a root, too. Thus homogeneous polynomials can be used to define subsets in projective space.

\begin{definition}
Let $\{P_k\}_k$ be a system of homogeneous polynomials in $n$ variables $\{z_i\}_i$. The subset $V = \{[\mathbf{z}] \in \mathbb{C}P^{n-1}: \forall\,k\,P_k(z_1,\dots,z_n) = 0\}$ is called a \emph{projective variety}.
\end{definition}

Note that any projective variety defines an algebraic variety in each affine chart $U_i$.

\section{Eigenvectors as fixed points}\label{fixed-points}

Consider an $n\times n$ complex matrix $M$. Its eigenvectors are elements of $\mathbf{z} \in\mathbb C^n$.
\begin{lemma}
The eigenvectors of $M$ are in one-to-one correspondence with nonzero solutions of the systems $\{\mathcal E_{ij}\}_{i,j}$ of polynomial equations in coordinates of $\mathbf{z}$ for all $i$ and $j$, where
$$
\mathcal E_{ij}:(M\mathbf{z})_j z_i = (M\mathbf{z})_i z_j.
$$
\end{lemma}

\begin{proof}
The system $\{\mathcal E_{ij}\}_{i,j}$ states that the tensor $M\mathbf{z}\otimes \mathbf{z} - \mathbf{z}\otimes M\mathbf{z}$ vanishes. This statement is equivalent to the statement that $M\mathbf{z}$ and $\mathbf{z}$ are collinear.
\end{proof}

As already noted, the eigenvectors of matrix $M$ are naturally identified with elements of the complex projective space $\mathbb{C}P^{n-1}$. From this point of view, the system $\{\mathcal E_{ij}\}_{i,j}$ defines a projective variety.

Fix an index $i$. To find the full set of eigenvectors in a single chart $U_i$ it is enough to solve a smaller set of equations.

\begin{lemma}\label{subsystem-equivalence}
Eigenvectors $\mathbf z$ of $M$ such that $[\mathbf z] \in U_i$ are in one-to-one correspondence with the solutions of the system of equations $\{\mathcal E_{ij}\}_{j:j\neq i}$ in the same chart.
\end{lemma}

\begin{proof}
We need to prove that for $[\mathbf{z}] \in U_i$ the two following propositions are equivalent: $(i)$ $\mathbf{z}$ is a root of $(M\mathbf{z})_j z_i = (M\mathbf{z})_i z_j$ for $j \neq i$ and $(ii)$ $\exists \eigval \in \mathbb C$ $M\mathbf{z} = \eigval \mathbf{z}$.

The $(i)$ $\Rightarrow$ $(ii)$ direction. For $[\mathbf{z}] \in U_i$ it is enough to consider only vectors with $z_i = 1$. Then from $\{\mathcal E_{ij}\}_{j:j\neq i}$ we conclude that $(M\mathbf{z})_j = (M\mathbf{z})_i z_j$ for all $j \neq i$. If we now denote the common factor in these expressions as $\eigval = (M\mathbf{z})_i$, then $(ii)$ immediately follows.

The $(ii)$ $\Rightarrow$ $(i)$ direction. Consider an eigenvector $\mathbf{z}$ of $M$ such that $z_i \neq 0$, and thus $[\mathbf{z}] \in U_i$. From the eigenvalue
equation $M\mathbf{z}=\eigval \mathbf{z}$ we can express its eigenvalue as $\eigval=(M \mathbf{z})_i/z_i$; inserting this back into $M \mathbf{z}=\eigval \mathbf{z}$ shows that $\mathbf{z}$ is a root of the system $\{\mathcal E_{ij}\}_{j:i\neq j}$.
\end{proof}

Denoting $V_i$ the projective variety defined by the system $\{\mathcal E_{ij}\}_{j:j\neq i}$ for a fixed $i$ and recalling that the set $\{U_i\}_i$ forms an atlas we conclude that the set of eigenvectors of $M$ can be identified with the set $\cup_i(V_i\cap U_i)$, which is in fact the projective variety $\cap_i V_i$. Since eigenvectors generically have non-zero coordinates in all directions, each component $V_i\cap U_i$ typically contains the complete set of eigenvectors.

It should be noted, however, that $V_i$ taken alone may have points unrelated to the eigenvectors of $M$, but these additional points can only be at infinity in $U_i$ by Lemma~\ref{subsystem-equivalence}. Indeed, consider points of $V_i$ such that $z_i = 0$. At the very least they include points that are given by $(M\mathbf{z})_i = 0$ in addition to $z_i = 0$. Thus, in general, there is a whole ($n-3$)-dimensional complex projective subspace of such solutions in $\mathbb CP^{n-1}$ (corresponding to a ($n-2$)-dimensional complex hyperplane in the affine space $\mathbb C^n$).

\begin{algorithm}[t]
    \caption{One eigenvector of a near-diagonal matrix $M$}\label{single_line_algo}
    \begin{algorithmic}[1]
     
        \STATE{Choose a partition $M = D + \Delta$ with $D$ diagonal}
        \STATE{Compute $\mathbf{g}_i \leftarrow ((D_{jj}-D_{ii})^{-1})_{j:j\neq i}$ and set $(\mathbf g_i)_i =0$}
        \STATE Initialize $\mathbf{z} \leftarrow \mathbf{e}_i$
        \WHILE{$\vert\mathbf z - \mathbf{f}_i(\mathbf z)\vert/\vert\mathbf z\vert >\eta$ for some tolerance $\eta$}        \STATE{$\mathbf z \leftarrow \mathbf{f}_i(\mathbf z)$ with $\mathbf{f}_i$ defined by \cref{single_line_dyn}}
        \ENDWHILE
        \STATE{Return eigenvector $\mathbf{z}$}
        \STATE{Return eigenvalue $\eigval_i = (M \mathbf{z})_i$}
    
    \end{algorithmic}
    \end{algorithm}

We now further assume a partitioning
$$M=D+\Delta$$
with a diagonal part $D$ and the residual part $\Delta$. The motivation comes from physics, where the diagonal part $D$ often consists of unperturbed
eigenvalues $\evzero_i=D_{ii}$ and $\Delta$ represents perturbations in the so-called perturbation theory. In practice $D$ can be taken as the diagonal elements of $M$, although different partitionings can sometimes be more appropriate \cite{Surj_n_2004}. Provided the $\evzero_i$'s are all simple (non-degenerate, that is pairwise different), we can rewrite the polynomial system $\{\mathcal E_{ij}\}_{j:j\neq i}$ for a particular fixed $i$ as
$$
(D\mathbf{z} + \Delta\mathbf{z})_j z_i = (D\mathbf{z} + \Delta\mathbf{z})_i z_j\ \textrm{for}\ j\neq i,
$$
or equivalently
$$
\evzero_j z_j z_i + (\Delta\mathbf{z})_j z_i = \evzero_i z_i z_j + (\Delta\mathbf{z})_i z_j\ \textrm{for}\ j\neq i,
$$
which yields
$$
z_i z_j=(\evzero_j-\evzero_i)^{-1}\left(z_j(\Delta \mathbf{z})_i-z_i(\Delta \mathbf{z})_j\right)\ \textrm{for}\ j\neq i.
$$
In the affine subspace $U_i$ this gives (by setting $z_i = 1$)
$$
z_j=(\evzero_j-\evzero_i)^{-1}\left(z_j(\Delta \mathbf{z})_i-(\Delta \mathbf{z})_j\right)\ \textrm{for}\ j\neq i.
$$
Thus, the eigenvectors of $M$ in $U_i$ can be identified with the solutions of the fixed point equation $\mathbf z = \mathbf{f}_i(\mathbf z)$ with $\mathbf{f}_i:\mathbb C^n\to \mathbb C^n$ the map 
\begin{equation}
	\label{single_line_dyn} 
	\mathbf{f}_i(\mathbf{z}) = \mathbf{e}_i + \mathbf{g}_i\circ(\mathbf{z} (\Delta \mathbf{z})_i - \Delta \mathbf{z})
\end{equation}
where $\mathbf{e}_i$ is the $i$-th standard basis vector of $\mathbb{C}^n$, $\mathbf{g}_i$ is the $i$-th column of the \emph{inverse gaps matrix} $G$ with components $G_{jk} = (\mathbf g_k)_j = (\evzero_j-\evzero_k)^{-1}$ for $j\neq k$ and $G_{jj} = (\mathbf g_j)_j = 0$. Here $\circ$ denotes the component-wise product of vectors in the standard basis. To obtain the eigenvector of $M$ closest to the basis vector $\mathbf{e}_i$, we can try to solve \cref{single_line_dyn} by fixed-point iteration; this is the basic idea of our method, presented in Algorithm \ref{single_line_algo}. As noted, each iteration of $\mathbf f_i$ consists of a single multiplication of the present vector by the perturbation $\Delta$, followed by element-wise multiplication by the vector $\mathbf{g}_i$. The resulting solution will be automatically normalized such that its $i$-th coordinate is equal to 1.

The same iterative technique can be used to compute all eigenvectors of $M$ in parallel (Algorithm \ref{full_spectrum_algo}). For this it suffices to bundle all $n$ candidate
eigenvectors for each $i$ into a matrix $\Z$ and apply the map $\mathbf{f}_i$ to the $i$-th column of $\Z$. This corresponds to the matrix map

\begin{equation}
\label{dyn} 
F(\Z)\equiv I+G\circ\Big(\Z\,\mathcal{D}(\Delta \Z) - \Delta \Z\Big),
\end{equation}
where $\,\circ\,$ denotes the Hadamard (element-wise) product of matrices and $\mathcal{D}(X)$ is the diagonal matrix built with the diagonal elements of matrix $X$. Starting from $\Z^{(0)} = I$, we obtain a sequence of matrices $\Z^{(k)} = F(\Z^{(k-1)})$ whose limit as $k\to\infty$, if exists, is the full set of eigenvectors. We call this approach \textit{iterative perturbation theory} (IPT).

\begin{algorithm}[t]
\caption{Full eigendecomposition of a near-diagonal matrix $M$}
\begin{algorithmic}[1]
     \label{full_spectrum_algo}
    \STATE{Choose a partition $M = D + \Delta$ with $D$ diagonal}
    \STATE{Compute $G \leftarrow ((D_{ii}-D_{jj})^{-1})_{i\neq n}$ and set $g_{ii} =0$}
    \STATE{Initialize $\Z \leftarrow I$}
    \WHILE{$\Vert \Z -F(\Z)\Vert/\Vert \Z\Vert > \eta $ for some tolerance $\eta$}
        \STATE{$\Z \leftarrow F(\Z)$ with $F$ defined by \cref{dyn}}
    \ENDWHILE
    \STATE{Return eigenmatrix $\Z$}
    \STATE{Return eigenvalues $\Eigval = \mathcal{D}(M \Z)$}
\end{algorithmic}
\end{algorithm}

\section{Convergence and divergence}\label{convergence}

In this section we look at the convergence of fixed-point iteration for the map \cref{dyn}. In a nutshell, the off-diagonal elements $\Delta$ must be small compared to the diagonal gaps $\evzero_i-\evzero_j$, a typical condition for eigenvector perturbation theory \cite{kato1995}. 

\subsection{A sufficient condition for convergence}

Let $\|\cdot\|$ denote the spectral norm of a matrix, i.e. its largest singular value. Let $M \in \mathbb C^{n\times n}$ be a matrix. Let $M = D + \Delta$ be its partition into a diagonal matrix $D$ and the residual matrix $\Delta$ such that the corresponding to $D$ matrix of inverse gaps $G$ is defined, which implies that all diagonal elements of $D$ are pairwise different. Let $F:\mathbb C^{n\times n} \to \mathbb C^{n\times n}$ be the mapping defined by $G$ and $\Delta$ as in the previous section.

\begin{theorem}\label{main-th}
If

\begin{equation}
\label{condition-unique}
\|G\|\|\Delta\| < 3 - 2\sqrt{2},
\end{equation}
then the dynamical system defined by iterations of $F$ and $A^{(0)} = I$ converges to a unique, asymptotically stable fixed point in the ball $B_{\sqrt{2}}(I)$.
\end{theorem}

\begin{proof}

The proof uses the Banach fixed-point theorem based on the $\|\cdot\|$ norm, which is sub-multiplicative with respect to both the matrix and Hadamard products~\cite{johnson2012}.

First, the estimate

\[\Vert F(\Z)-I\Vert\leq \Vert G\Vert \Vert\Delta\Vert (\Vert \Z\Vert+\Vert \Z\Vert^2)\]
implies that \(F\) maps a closed ball \(B_r(I)\) of radius \(r\) centered on \(I\) onto itself whenever \(\Vert G\Vert\Vert\Delta\Vert\leq r/[(1+r)(2+r)]\).  Next, from {\cref{dyn}} we have the estimate

\[
\Vert F(\Z)-F(\Z')\Vert\leq \Vert G\Vert \Vert\Delta\Vert\,(1+\Vert \Z+  \Z'\Vert)\, \Vert \Z -  \Z'\Vert.
\]
Hence, \(F\) is contracting
in \(B_r(I)\) provided \(\Vert G\Vert\Vert\Delta\Vert < 1/[1+2(1+r)]\). When both conditions on $\|G\|\|\Delta\|$ hold the Banach fixed-point theorem implies
that \(\Z^{(k)}=F^k(I)\) converges exponentially to a unique fixed point $\Z^*$
within $B_r(I)$ as \(k\to\infty\). Choosing the optimal
radius
$$\underset{r>0}{\textrm{argmax}}\min\,\left(\frac{r}{(1+r)(2+r)},\frac{1}{1+2(1+r)}\right)=\sqrt{2},$$
we see that \(\|G\|\Vert\Delta\Vert<3-2\sqrt{2}\approx 0.17\) guarantees convergence to the fixed point $\Z^*$.
\end{proof}

The set of solutions of the equation $\Z = F(\Z)$ contains matrices composed of all combinations of eigenvectors of $M$ with the appropriate normalization (the $i$-th coordinate of the column at the $i$-th position is set to 1). Some solutions may contain several repeated eigenvectors. Such solutions are rank-deficient, viz. $\operatorname{rank} \Z < n$. In principle, there is a danger that the iterative algorithms converges to one such solution with a loss of information about the eigenvectors of $M$ as a result.
The following theorem guarantees that under condition (\ref{condition-unique}) this does not happen.

\begin{theorem}
\label{theorem-rank}
Let condition (\ref{condition-unique}) hold for a partition of matrix $M$ and $\Z^*$ be the unique fixed point of $F$ in $B_{\sqrt{2}}(I)$. Then $\Z^*$ has full rank.
\end{theorem}

\begin{proof}
The proof uses additional lemmas and one theorem provided in \cref{appendix:full_rank}.

A square matrix is rank-deficient if either two of its columns are collinear or more than two of its columns are linearly dependent but not pairwise collinear. For $\Z^*$ the latter is only possible if the corresponding eigenvectors belong to an eigenspace of $M$ of dimension higher than one, which is ruled out by Lemma~\ref{lemma-eigenspace}. By Lemma~\ref{lemma-defective}, $\Z^*$ does not contain any defective eigenvectors of $M$. Then the rank can be lost only by a repetition of an eigenvector (with renormalization) corresponding to an eigenvalue of multiplicity one.

Embed $M$ into the family $M_\parameter = D + \parameter \Delta$ with $\parameter \in [0,1]$, so that $M_1 = M$ and $M_0 = D$. Consider a partition for each member of the family $M_\parameter = D_\parameter + \Delta_\parameter$ such that $D_\parameter = D$ and $\Delta_\parameter = \parameter\Delta$. Let $G_\parameter = G$ be the inverse gaps matrix of $D_\parameter$ and $F_\parameter$ be the mapping defined by $G_\parameter$ and $\Delta_\parameter$ according to (\ref{dyn}) with a substitution of $G$ by $G_\parameter$ and $\Delta$ by $\Delta_\parameter$. It follows that $\|G_\parameter\|\|\Delta_\parameter\| < 3 - 2\sqrt{2}$ holds for all $\parameter$. Thus, according to Theorem~\ref{main-th}, each $F_\parameter$ has a unique fixed point $\Z^*_\parameter$ in $B_{\sqrt{2}}(I)$. Furthermore, by Lemmas~\ref{lemma-eigenspace} and \ref{lemma-defective} we can be sure that none of the columns of $\Z^*_\parameter$ for each value of $\parameter$ corresponds either to an eigenvector with algebraic multiplicity higher than one or lays in an eigenspace of $M_\parameter$ with dimension higher than one. Thus they all correspond to eigenvalues that are simple roots of the characteristic polynomial of $M_\parameter$. Then all projective points that induce columns of $\Z^*_\parameter$ are differentiable functions of $\parameter$ by Theorem~\ref{theorem-continuous}. It means that if some two columns of $\Z^* = \Z^*_1$ are collinear (and thus correspond to the same projective point), the same columns stay collinear in $\Z^*_\parameter$ for all $\parameter$. But none of the columns of $I = \Z^*_0$ are collinear. We conclude that $\Z^*$ cannot have collinear columns.
\end{proof}

\begin{corollary}
If $M$ has eigenvalues with multiplicity higher than one, then for any partition $\tilde M = \tilde D + \tilde\Delta$ of any matrix $\tilde M$ similar to $M$, where $\tilde D$ is diagonal such that the corresponding inverse gaps matrix $\tilde G$ is defined, the relation $\|\tilde G\| \|\tilde\Delta\| \geq 3 - 2\sqrt{2}$ holds.
\end{corollary}

The following subsections are dedicated to aspects of IPT related to quantum-mechanical perturbation theory (sec. \ref{IPT-RS}) and dynamical systems theory (sec. \ref{2d-example}). They may be omitted by the non-specialist reader.

\begin{figure*}
    \includegraphics[width=\textwidth]{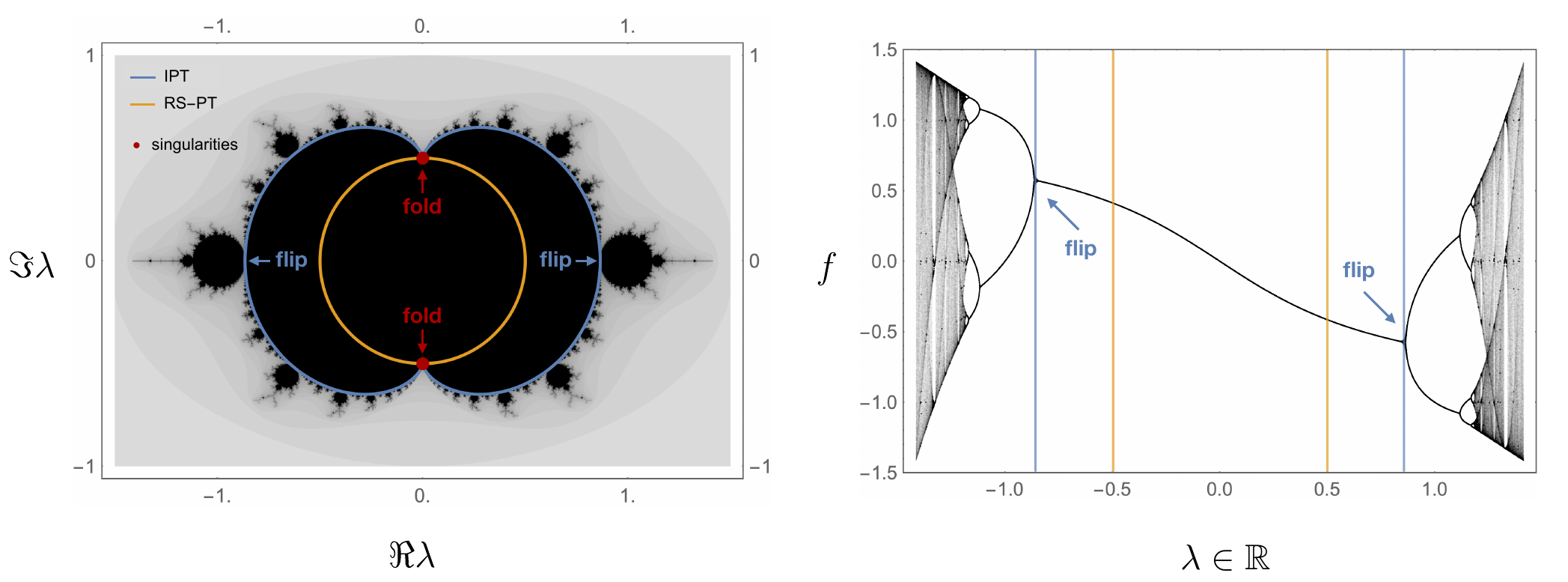}
    \caption{Convergence of IPT in the two-dimensional
    example \cref{2by2}. Left: In the
    complex $\parameter$-plane, RS perturbation theory (RS-PT)
    converges inside a circle of radius $1/2$ (orange line)
    bounded by the exceptional points $\pm i/2$ where eigenvalues
    have branch-point singularities and $M$ is not
    diagonalizable. Dynamical perturbation theory (IPT) converges inside
    the domain bounded by the blue cardioid, which is larger---especially along the
    real axis, where there is no singularity. Outside this domain, the map
    can converge to a periodic cycle, be chaotic or diverge to infinity, following flip bifurcations (along the real axis)
    and fold bifurcations (at the singularities of the  blue curve). The domain where the map
    remains bounded (black area) is a conformal transformation of the
    Mandelbrot set. Right: The bifurcation diagram for the quadratic
    map $f$ along the real $\parameter$-axis
    illustrates the period-doubling route to chaos as $\parameter$
    increases away from $0$ (in absolute value). Orange and
    left vertical lines indicate the boundary of the convergence domains of
    RS-PT and IPT respectively. }
    \label{2d}
    \end{figure*}

\subsection{Contrast with Rayleigh-Schr\"odinger perturbation theory}\label{IPT-RS}

It is interesting to constrast the present iterative method with conventional Rayleigh-Schr\"odinger (RS)
perturbation theory, where the eigenvectors of a parametric matrix $M = D + \parameter \Delta$ are
sought as power series 
in $\parameter$, \emph{viz.} $\mathbf z=\sum_{\ell}\parameter^\ell \mathbf z^{[\ell]}$, where vector-terms (so called corrections) $\mathbf z^{[\ell]}$ do not depend on $\parameter$. Provided that $D$ has distinct diagonal entries, it is indeed possible to express the matrix of eigenvectors $\Z$ (with the same normalization as in Section~\ref{fixed-points} and with the order of eigenvectors such that $\Z \to I$ as $\parameter \to 0$) as a power series $\Z = \sum_\ell \parameter^\ell \Z^{[\ell]}$, which converges in some disk around $\parameter = 0$ \cite{kato1995}. Then the order-$k$ approximation of $\Z$ takes the form
$\Z^{(k)}_{\textrm{RS}}=\sum_{\ell=0}^k \parameter^\ell \Z^{[\ell]}$. The matrix
corrections $\Z^{[\ell]}$ are obtained from $\Z^{[0]}=I$
via the recursion
(\cref{appendix:RSPT})

\begin{equation}\label{ARS-alpha} 
\Z^{[\ell]}=G\circ\left(\sum_{s=0}^{\ell-1} \Z^{[\ell-1-s]}\,\mathcal{D}(\Delta \Z^{[s]}) -\Delta \Z^{[\ell-1]}\right).
\end{equation}
The iterative
scheme $ \Z^{(k)}$ completely contains this RS series in the
sense that $\Z^{(k)}=\Z^{(k)}_{\textrm{RS}}+\mathcal{O}(\parameter^{k+1})$; this can be seen by induction (\cref{appendix:IPT_RSPT}). In other words, we can recover the usual perturbative expansion
of $\Z$ to order $k$ by
iterating $k$ times the map $F$ and
dropping all terms $\mathcal{O}(\parameter^{k+1})$. Moreover, the parameter whose
smallness determines the convergence of the RS series is the product of
the perturbation magnitude $|\parameter|$ with the inverse diagonal
gaps $\|G\|$ \cite{kato1995}, just as it determines the
contraction property of $F$.

But IPT also differs from
the RS series in two key ways. First, the complexity of each iteration
is constant (essentially just one matrix product with $\Delta$), whereas
computing the RS corrections $\Z^{[\ell]}$ involves the sum of
increasingly many matrix products. Second, not being defined as a power
series, the convergence of $ \Z^{(k)}$ when $k\to\infty$
is not \emph{a priori } restricted to a disk in the
complex $\parameter$-plane. Together, these two differences
suggest that IPT has the potential to converge faster,
and in a larger domain, than RS perturbation theory. This is what we now
examine, starting from an elementary but explicit example.

\subsection{An explicit $2\times 2$ example}\label{2d-example}

To build intuition, let us consider the parametric matrix 
\begin{equation}
\label{2by2} 
M =
\begin{pmatrix}
   0 & \parameter \\
   \parameter & 1  
\end{pmatrix}
.
\end{equation}
This matrix has eigenvalues $\eigval_{\pm}=(1\pm\sqrt{1+4\parameter^2})/2$, both of which are analytic inside the disk $\vert\parameter\vert<1/2$ but have branch-point singularities at $\parameter=\pm i/2$. (These singularities are exceptional points, \textit{i.e.} $M$ is not diagonalizable for these values.) Because the RS series is a power series, these imaginary points contaminate its convergence also on the real axis, where no singularity exists: $\Z_{\textrm{RS}}$ diverges for any value of $\parameter$ outside the disk of radius $1/2$, and in particular for real $\parameter>1/2$ .

Considering instead our iterative scheme, one easily computes $$\Z^{(k)}=
\begin{pmatrix}
    1 & -f^k(0) \\
    f^k(0) & 1
    \end{pmatrix}
,$$
where $f(x)=\parameter(x^2-1)$ and the superscripts indicate $k$-fold iterates. This one-dimensional map has two fixed points at $x^*_{\pm}=\eigval_{\pm}/\parameter$. Of these two fixed points $x^*_+$ is always unstable, while $x^*_-$ is stable for $\parameter \in(-\sqrt{3}/2, \sqrt{3}/2)$ and loses its stability at $\parameter=\pm\sqrt{3}/2$ in a flip bifurcation. At yet larger values of $\parameter$, the iterated map $f^k$---hence the fixed-point iterations $ \Z^{(k)}$---follows the period-doubling route to chaos familiar from the logistic map \cite{May_1976}. For values of $\parameter$ along the imaginary axis, we find that the map is stable if $\Im \parameter\in(-1/2,1/2)$ and loses stability in a fold bifurcation at the exceptional points $\parameter=\pm i/2$. The full domain of convergence of the system is strictly larger than the RS disk, as shown in \cref{2d}. We also observe that the disk where both schemes converge, the dynamical scheme does so with a better rate than RS perturbation theory: we check that $\vert f^k(0)-x^*_-\vert\sim\vert1-\sqrt{1+4\parameter^2}\vert^k=\mathcal{O}(\vert 2\parameter^2\vert^k)$, while the remainder of the RS decays as $\mathcal{O}(\vert 2\parameter\vert^k)$. This is a third way in which the dynamical scheme outperforms the usual RS series, at least in this case: not only is each iteration computationally cheaper, but the number of iterations required to reach a given precision is lower.

Although it is possible to analytically compute the convergence domain of our iterative scheme in this simple case, in practice this task is prohibitive even for very small matrices, as it can be seen on a slightly more complicated $3\times 3$ example. This example along with the computation of the convergence domain for the $2\times 2$ case can be found in \cref{further-examples}.

\section{Acceleration}\label{acceleration}

As a fixed-point problem, IPT is directly amenable to the broad set of fixed-point acceleration methods familiar from numerical analysis \cite{eyert1996comparative}. Generally speaking, acceleration techniques aim to reduce the number of iterations required for convergence by combining several iterates of the function in each update. 

We experimented with the well-known Anderson acceleration \cite{Walker_2011} and found that, indeed, IPT can be made to converge faster and in a larger domain that with simple Picard iteration. Here we use the more recent (but simpler) technique called Alternating Cyclic Extrapolation (ACX) \cite{lepage2021alternating}. This method consists in replacing line 4,5 in \cref{single_line_algo} (or similarly in \cref{full_spectrum_algo}) by \cref{ACX}. As will see below, IPT with ACX acceleration (IPT-ACX) is a very efficient approach to near-diagonal eigenvalue problems.

\begin{algorithm}[t]
    \caption{Alternating Cyclic Extrapolation \cite{lepage2021alternating}, with vector of orders $o = (3, 2, 3, 2, \cdots)$}\label{ACX}
    \begin{algorithmic}[1]
        \STATE{$i = 1$}
        \WHILE{$\vert\mathbf z - \mathbf{f}_i(\mathbf z)\vert/\vert\mathbf z\vert >\eta$ for some tolerance $\eta$}
        \STATE{$p \leftarrow o_k$}
        \STATE{$\Delta^0 \leftarrow \mathbf z$}
        \STATE{$\Delta^1 \leftarrow \mathbf{f}_i(\mathbf z) - \mathbf z$}   
        \STATE{$\Delta^2 \leftarrow \mathbf{f}_i^2(\mathbf z) - 2 \mathbf{f}_i(\mathbf z) + \mathbf z$}
        \IF{$p = 2$}
        \STATE{$\sigma \leftarrow \vert\langle\Delta^2, \Delta^1\rangle \vert/\Vert \Delta^2\Vert$}
        \STATE{$\mathbf z \leftarrow 2\sigma \Delta^1 + \sigma^2 \Delta^2$}
        \ELSIF{$p = 3$}
        \STATE{$\Delta^3 \leftarrow \mathbf{f}_i^3(\mathbf z) - 3\mathbf{f}_i^2(\mathbf z) + 3 \mathbf{f}_i(\mathbf z) - \mathbf z$}
        \STATE{$\sigma \leftarrow \vert\langle\Delta^3, \Delta^2\rangle \vert/\Vert \Delta^3\Vert$}
        \STATE{$\mathbf z \leftarrow 3\sigma \Delta^1 + 3\sigma^2 \Delta^2 + \sigma^3\Delta^3$}
        \ENDIF
        \ENDWHILE
    \end{algorithmic}
    \end{algorithm}

\section{Performance}\label{performance}

We investigated the performance of IPT with respect to classical algorithms using several test matrices. Our system consists of $2\times 64$ CPU cores (AMD Epyc 7702 at 2.0 GHz) and one A100 NVidia GPU.

\subsection{Few eigenvalues} To compute a small subset of eigenvalues of a near-diagonal matrix, the reference algorithms are preconditioned iterative methods, with $D^{-1}$ as preconditioner. (Krylov methods are much slower for these problems.) We compared IPT to the symmetric eigensolver PRIMME (PReconditioned Iterative MultiMethod Eigensolver) \cite{stathopoulos2010primme}, a package containing various implementations of Rayleigh Quotient Iteration (RQI), Generalized Davidson (GD, GD$+k$), Jacobi-Davidson (JDQR, JDQMR), and Locally Optimal Block Preconditioned Conjugate Gradient (LOBPCG).\footnote{We also tested SLEPc \cite{hernandez2005slepc} but, on our single-node system, PRIMME was faster.} PRIMME also includes a "dynamic" mode with shifts dynamically between GD$+k$ and JDQMR in an attempt to minimize the number of matrix-vector multiplications. 

We remind the reader that GD, JD and RQI all consist in expanding a small search subspace, orthonormalizing that subspace, and computing the eigendecomposition of $M$ in that subspace. Where they differ is how the subspace is expanded. GD uses the current residual vector $\mathbf r = \mathbf{f}(\mathbf z) - \mathbf z$ after preconditioning with $D^{-1}$. JD and RQI, by contrast, solve a linear ``correction'' equation (approximately or exactly, respectively) and expand the subspace with the correction vector; this inner iteration involves additional matrix-vector products per step compared to GD. We refer the reader to the documentation of PRIMME for implementation details. 

For the implementation of IPT and ACX we used the Julia programming language \cite{Julia-2017}. The code is freely available at \cite{IPT-code} (for IPT) and \cite{paper-code} (for the timings and figures). As already noted, when using PRIMME we set $D^{-1}$ as preconditioner. Other than this, we used the default settings for all methods.   

Our first example is from a real-world chemistry application. Full-configuration interaction (FCI) consists in computing the ground state (lowest-lying eigenvector) of the Hamiltonian operator of a molecule in a finite basis set \cite{ostlund1996}. Because  molecular orbitals can be computed approximately using self-consisted field approximations, the corresponding (symmetric) Hamiltonian matrix is near-diagonal, and suitable eigensolvers must make use of this information. (Indeed, Davidson's diagonally-preconditioned algorithm \cite{Davidson_1975} was introduced for this purpose.) 

Please note that, unlike standard iterative eigensolvers, IPT does not naturally target extremal eigenvalues; instead it computes perturbations the diagonal elements of $M$. In particular, to compute the lowest eigenvalue of a FCI matrix, we simply choose apply \cref{single_line_algo} for $i$ such that $M_{ii} = \min \textrm{diag}(M)$. 

We computed the ground state of the FCI matrix for water ($\textrm{H}_2$O) in the minimal basis set ``sto-3g" (with $n = 441$ in this example). (To compute its elements we used the quantum chemistry package PySCF \cite{PYSCF}.) \cref{water} (left) shows the convergence history (residual norm $\Vert \mathbf{f}(\mathbf{z}) - \mathbf{z}\Vert$ vs. number of matrix-vector products, matvecs) for various algorithms in PRIMME and for IPT(-ACX). In this example, both IPT and IPT-ACX converge to the desired tolerance (here $\eta = 10^{-12}$) with fewer matvecs that any of the PRIMME methods, with the latter about twice faster than the former. \cref{water} (right) shows the corresponding timings on our multi-core machine. 

\begin{figure}[t]
    \centering
    \includegraphics[width=.49\textwidth]{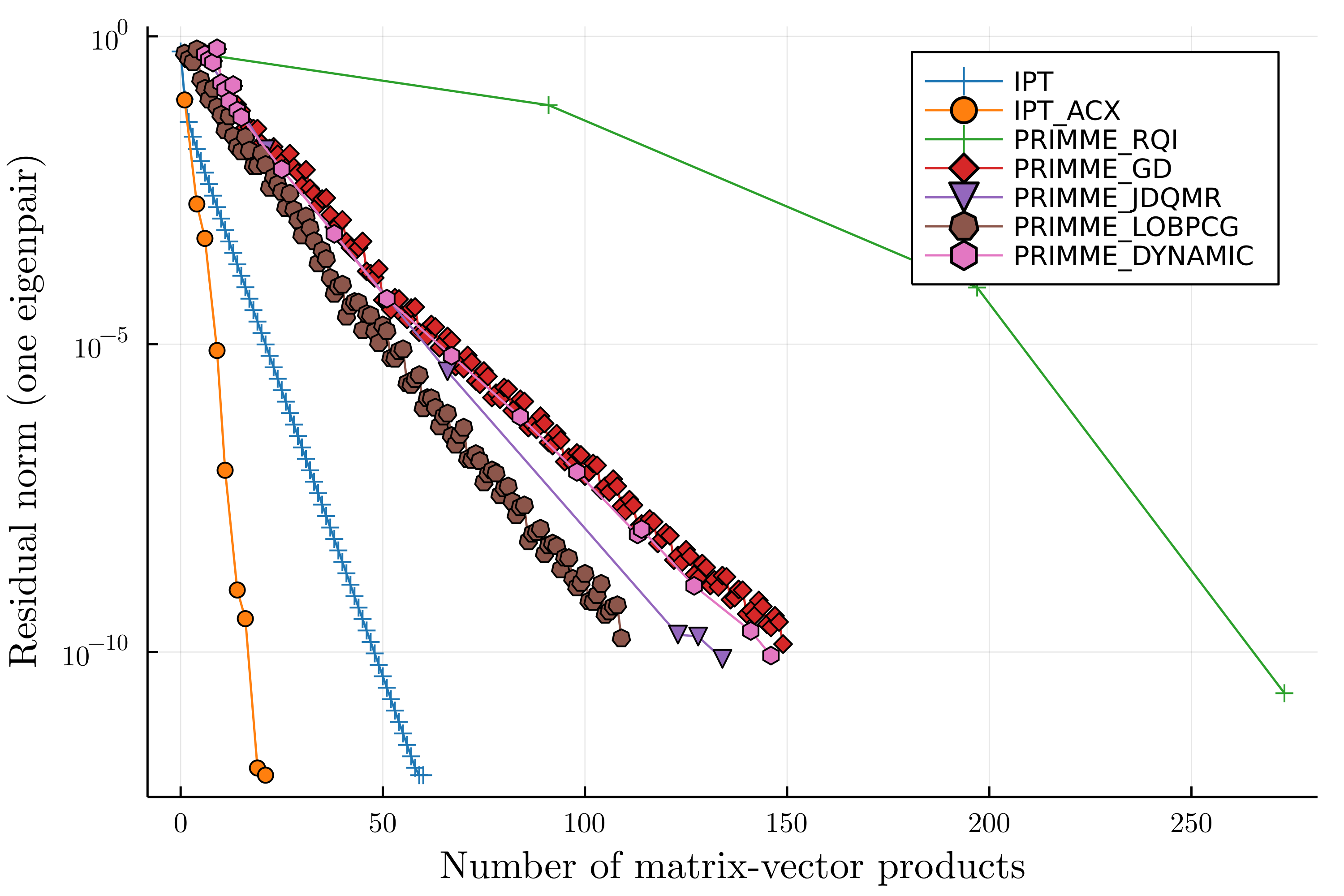}
    \hfill
    \includegraphics[width=.49\textwidth]{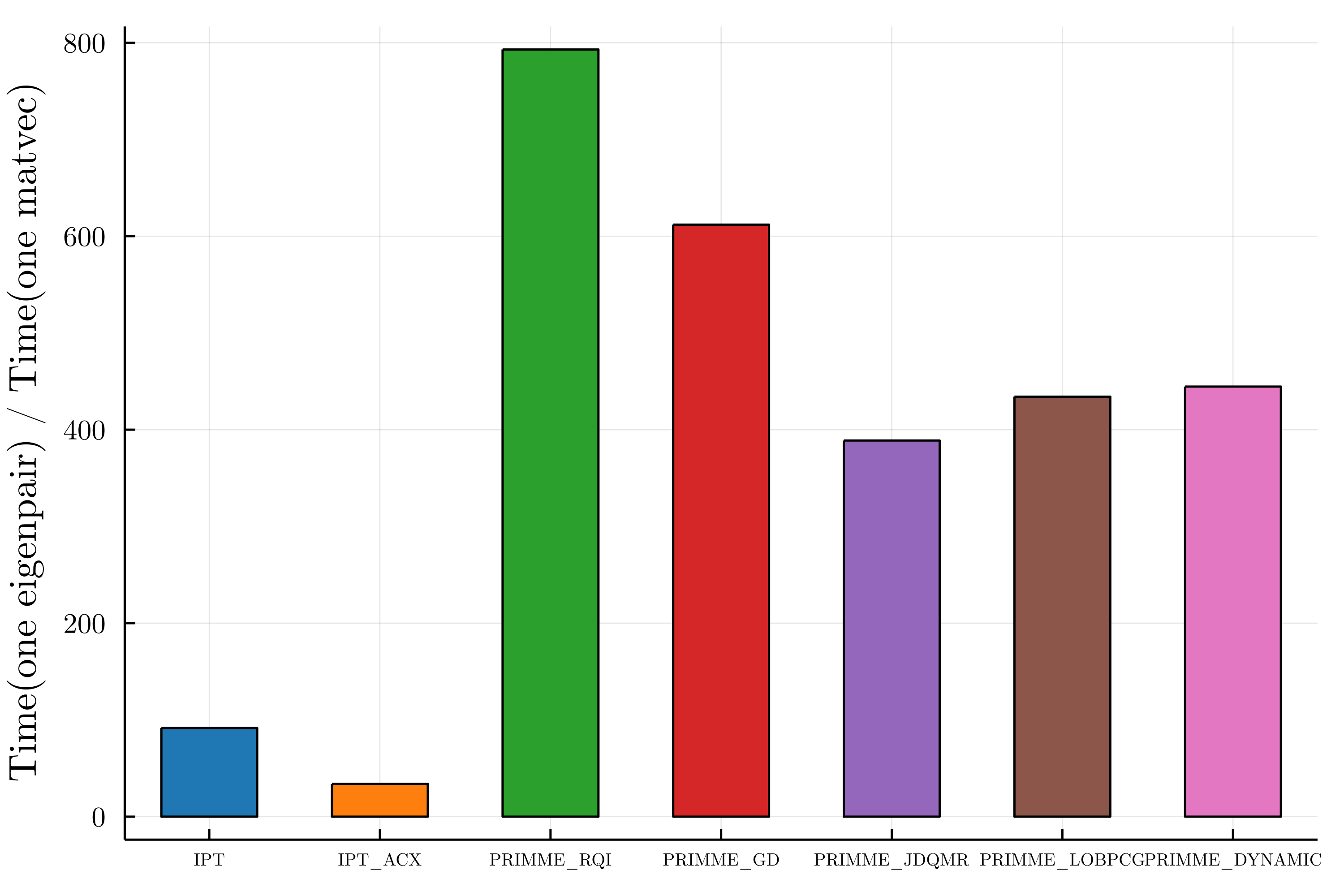}

    \caption{
        IPT vs. PRIMME preconditioned eigensolvers for the FCI computation of the ground state of the molecular Hamiltonian of $\textrm{H}_2$O ($n = 441$) in the minimal basis set ``sto-3g".  Left: convergence history, i.e. norm of the residual vector $\mathbf r = \mathbf{f}(\mathbf{z}) - \mathbf{z}$ vs. matvecs, showing best performance for IPT and IPT-ACX. Right: the corresponding timings on our multi-core CPU. 
    }
    \label{water}
\end{figure}

We performed the same calculations with a near-diagonal matrix considered by Morgan in his comparison of preconditioned eigensolvers \cite{morgan2000preconditioning}: $M$ is the tri-diagonal, symmetric matrix with $M_{ii} = i$ and $M_{i, i+1} = M_{i, i -1} = 0.5$. \cref{morgan} shows the results. Here IPT requires many more matvecs than Davidson methods, but IPT-ACX does not, converging equally fast as the fastest PRIMME method for this example, LOBPCG. But since LOBPCG involves additional steps per iteration (diagonalization in the subspace), IPT-ACX proves faster overall (\cref{morgan}, right panel).

\begin{figure}[t]
    \centering
    \includegraphics[width=.49\textwidth]{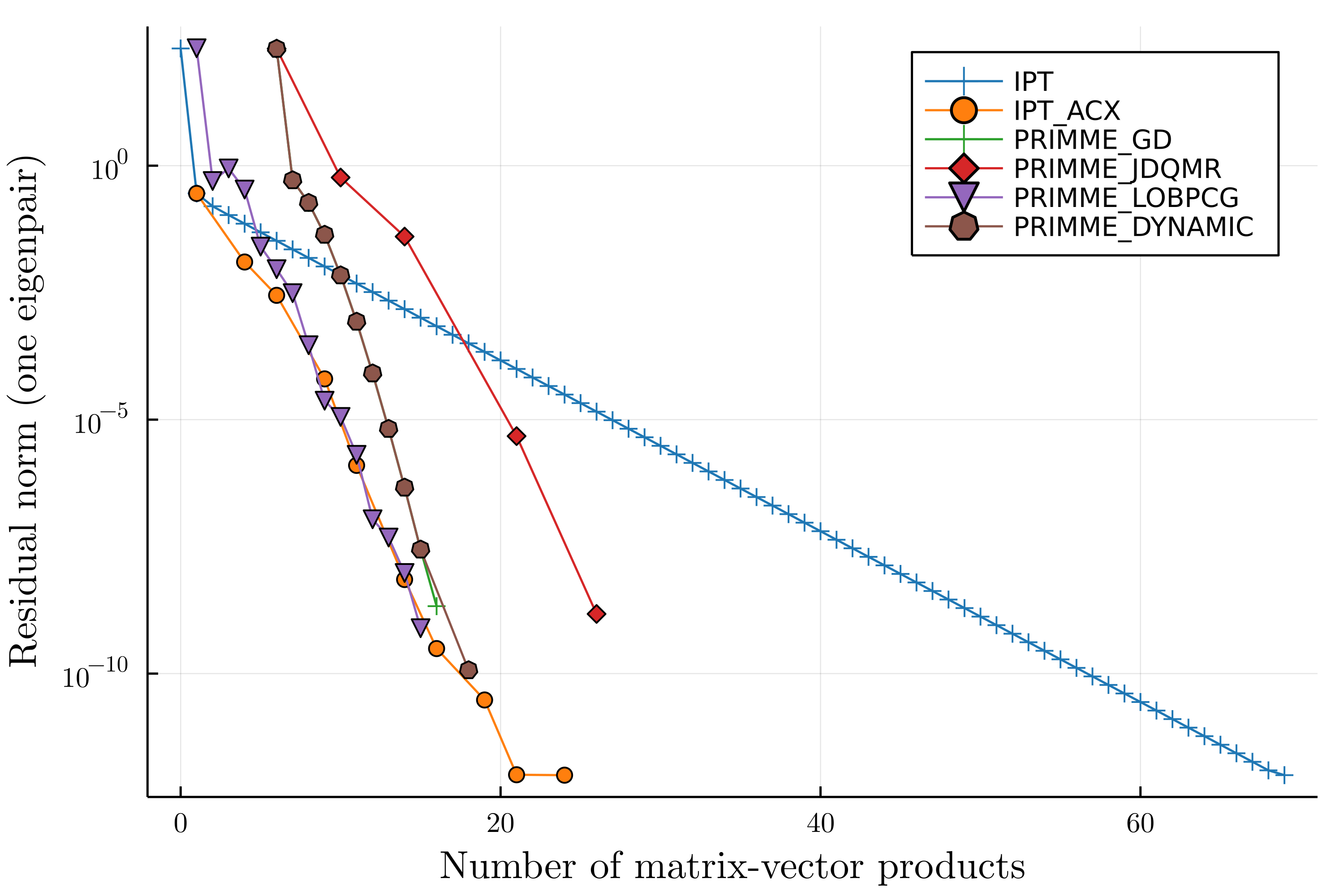}
    \hfill
    \includegraphics[width=.49\textwidth]{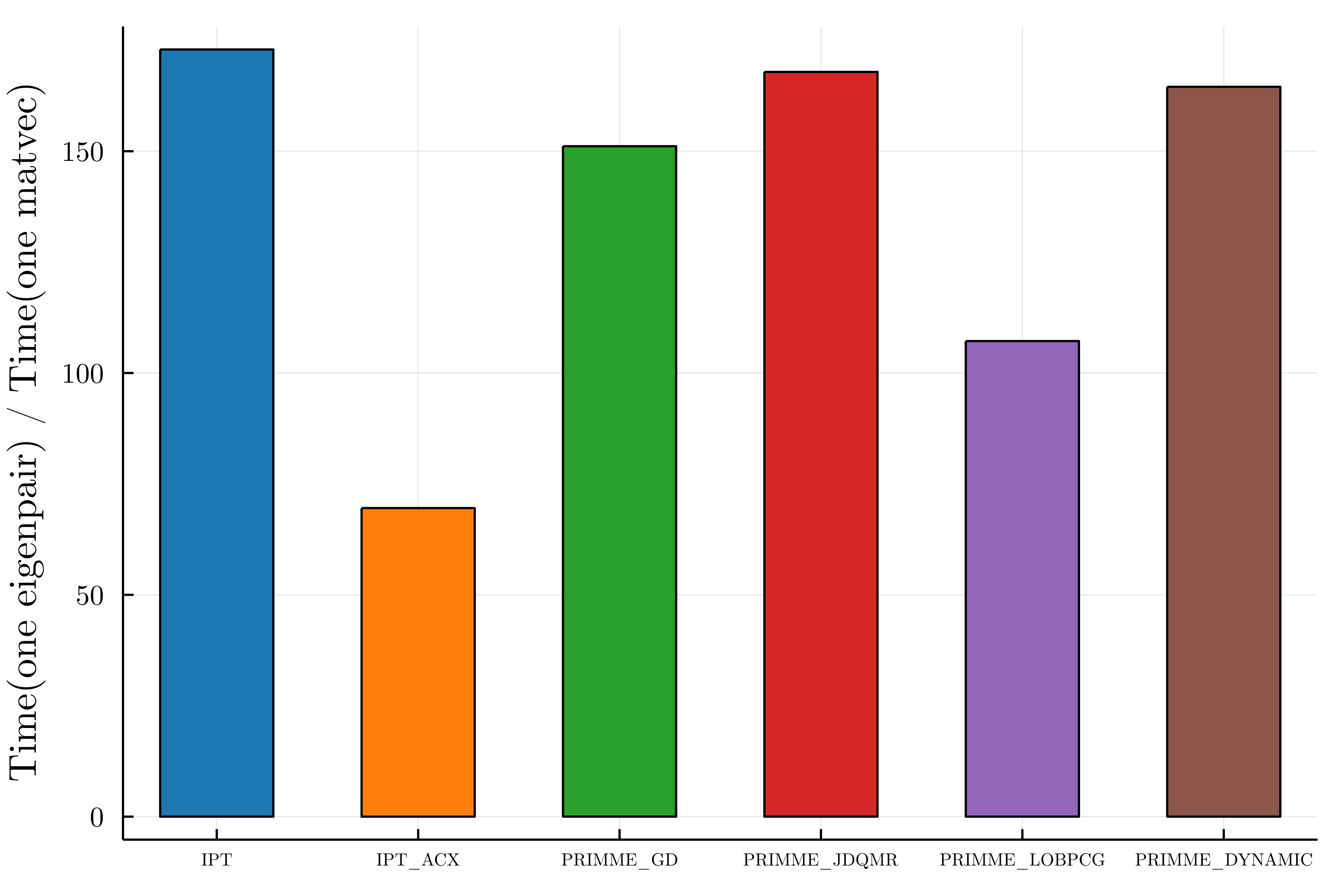}

    \caption{
        IPT vs. PRIMME preconditioned eigensolvers for the FCI computation of the ground state of the $5000\times 5000$ tri-diagonal, symmetric matrix considered in \cite{morgan2000preconditioning}.  Left: convergence history, i.e. norm of the residual vector $r = \mathbf{f}(\mathbf{z}) - \mathbf{z}$ vs. matvecs, showing that IPT-ACX requires as few matvecs as the best PRIMME eigensolver (in this case LOBPCG). Right: the corresponding timings on our multi-core CPU. 
    }
    \label{morgan}
\end{figure}

\begin{figure}[t]
    \centering
    \includegraphics[width=.9\textwidth]{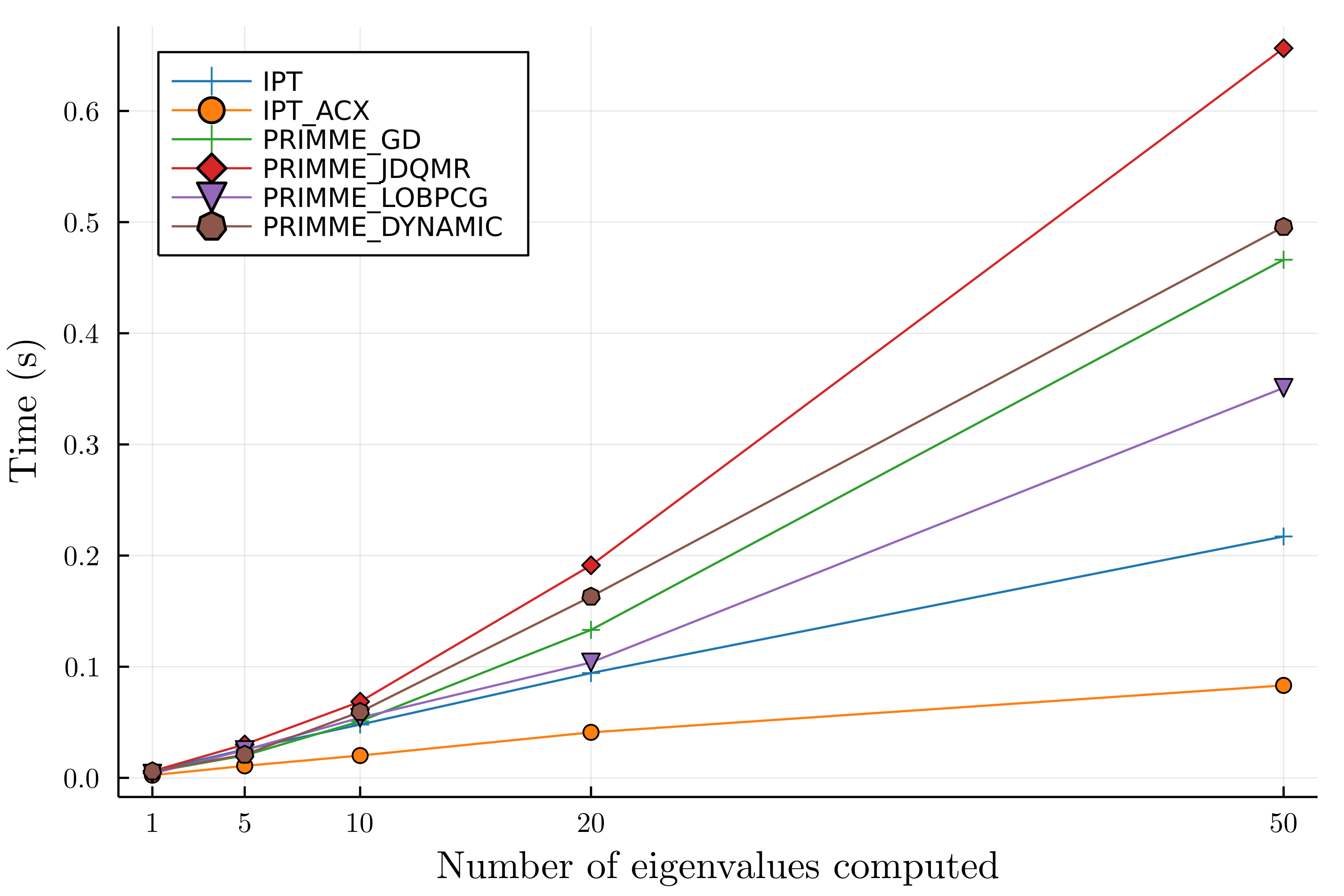}
    \caption{
        Time to compute the $k$ lowest eigenvalues of the tridiagonal Morgan matrix using IPT(-ACX) and PRIMME iterative eigensolvers. Thanks to IPT's perfect parallelism, the larger $k$, the larger its advantage over other methods which target extremal eigenvalues and progress inwards by deflation.  
    }
    \label{num_ev}
\end{figure}

But the advantage of IPT vis-a-vis standard iterative methods becomes more apparent when we request several eigenpairs rather than just one. Unlike other methods, IPT computes eigenpairs in parallel (corresponding to the columns of $Z$) rather than sequentially, after deflation. The benefit of this parallelism is on display in \cref{num_ev}, where the timing to compute $k$ eigenvalues is compared across algorithms. For $k=50$, IPT-ACX is already several times faster than the fastest iterative method, and this gap only increases with $k$.

Finally we show in \cref{fail} how IPT fails when off-diagonal elements become too large compared to diagonal gaps. For this we consider a modification of the Morgan matrix with $M_{ii} = i$ and $M_{i, i+1} = M_{i, i -1} = \varepsilon$ and increase $\varepsilon$.

\subsection{Full spectrum} 

Next we consider the problem of computing all eigenpairs of a large, dense matrix of the form 
\begin{equation}\label{dense_matrix}
    M = \textrm{diag}(i)_{1 \leq i\leq n} + \varepsilon R_n
\end{equation}
where $R_n$ is a $n\times n$ matrix with uniformly distributed random entries in $[0,1]$. For a symmetric matrix with similar properties we consider $S = (M + M^t)/2$.

\begin{figure}[t]
    \centering
    \includegraphics[width=.9\textwidth]{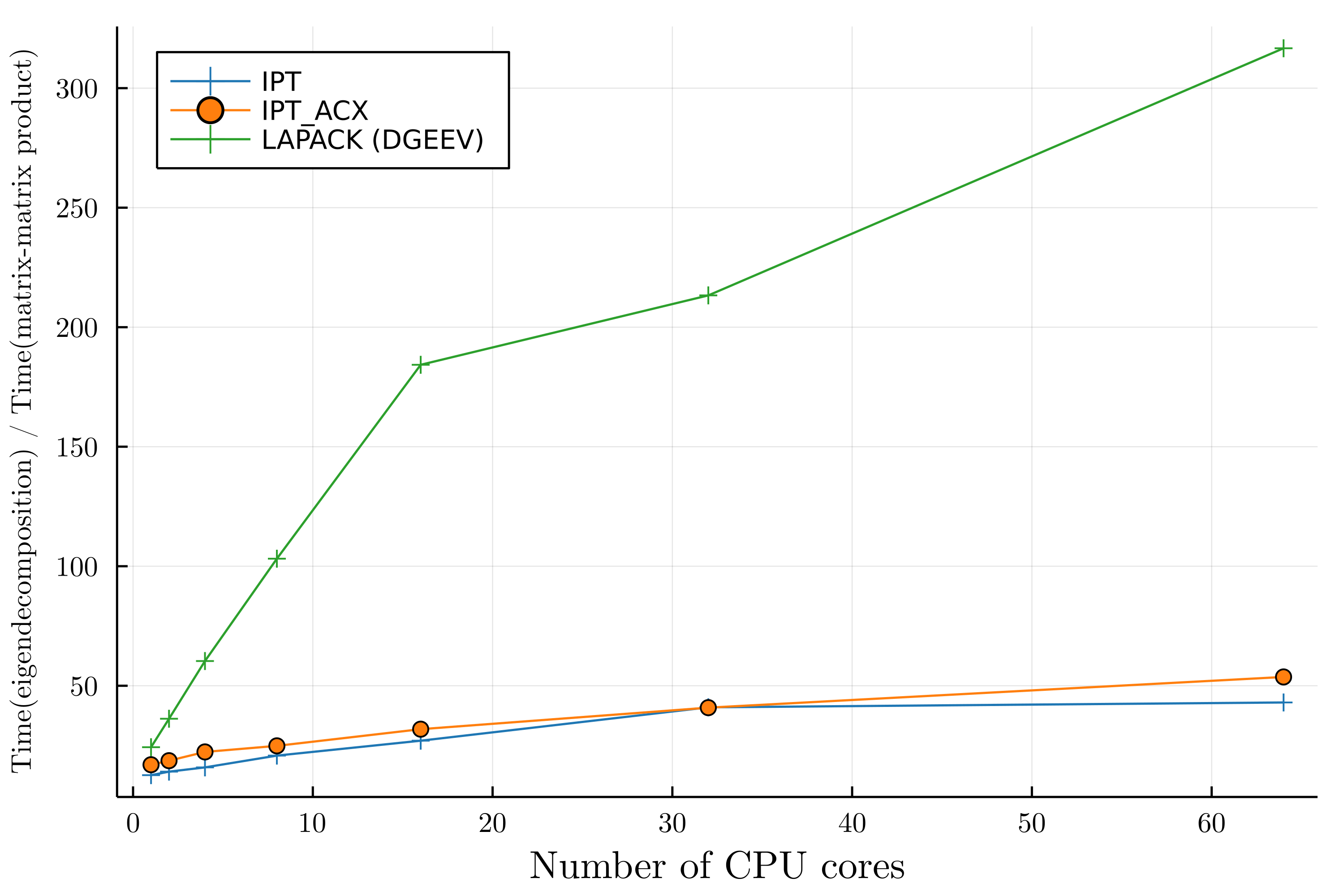}
    \caption{
       Time to compute the complete eigendecomposition of a dense, non-symmetric matrix with random entries scaled by $\varepsilon = 0.01$ compared to the time for one matrix-matrix multiplication (BLAS 3). Because IPT is perfectly parallel, this ratio hardly increases with the number of CPU cores used, in stark contrast with GEEV.
    }
    \label{parallel}
\end{figure}

Dense eigenproblems are normally solved with the QR or divide-or-conquer algorithm, implemented in the Linear Algebra PACKage (LAPACK) under the name (GEEV). On NVidia GPUs, the CUSOLVER \cite{c1998} package provides an efficient version of divide-and-conquer for symmetric matrices (SYEVD). As already emphasized, these methods do not take advantage of any particular structure in $M$ except symmetry; in particular they do not perform better on sparse or near-diagonal matrices than on general matrices. 

We timed multi-CPU and GPU implementations of IPT to these methods for matrices \cref{dense_matrix} with $\varepsilon = 10^{-2}$, double precision ($\eta = 10^{-12}$), and increasingly large dimension $n$. (Both IPT and IPT-ACX diverge for $\varepsilon \gtrsim .05$) compared with these reference routines (\cref{timings}). The result is that IPT is up to 2 orders of magnitude more efficient than the non-symmetric routine GEEV, and also more efficient than the (much faster) symmetric routine SYEVD.

The reason for this enhanced performance on near-diagonal problems is again due to the better parallelism of IPT. For the full-spectrum problem, IPT relies on matrix-matrix products which are efficiently parallelized with BLAS 3; the Hessenberg reduction with underlies direct methods does not benefit from such parallelism. \cref{parallel} shows the timing of eigendecompositions vs. that of matrix-matrix multiplications as a function of the number of CPU cores used. While this ratio increases with GEEV (indicating worse parallelism than BLAS 3), it does not with IPT (as expected).

\begin{figure}[t]
    \centering
    \includegraphics[width=.9\textwidth]{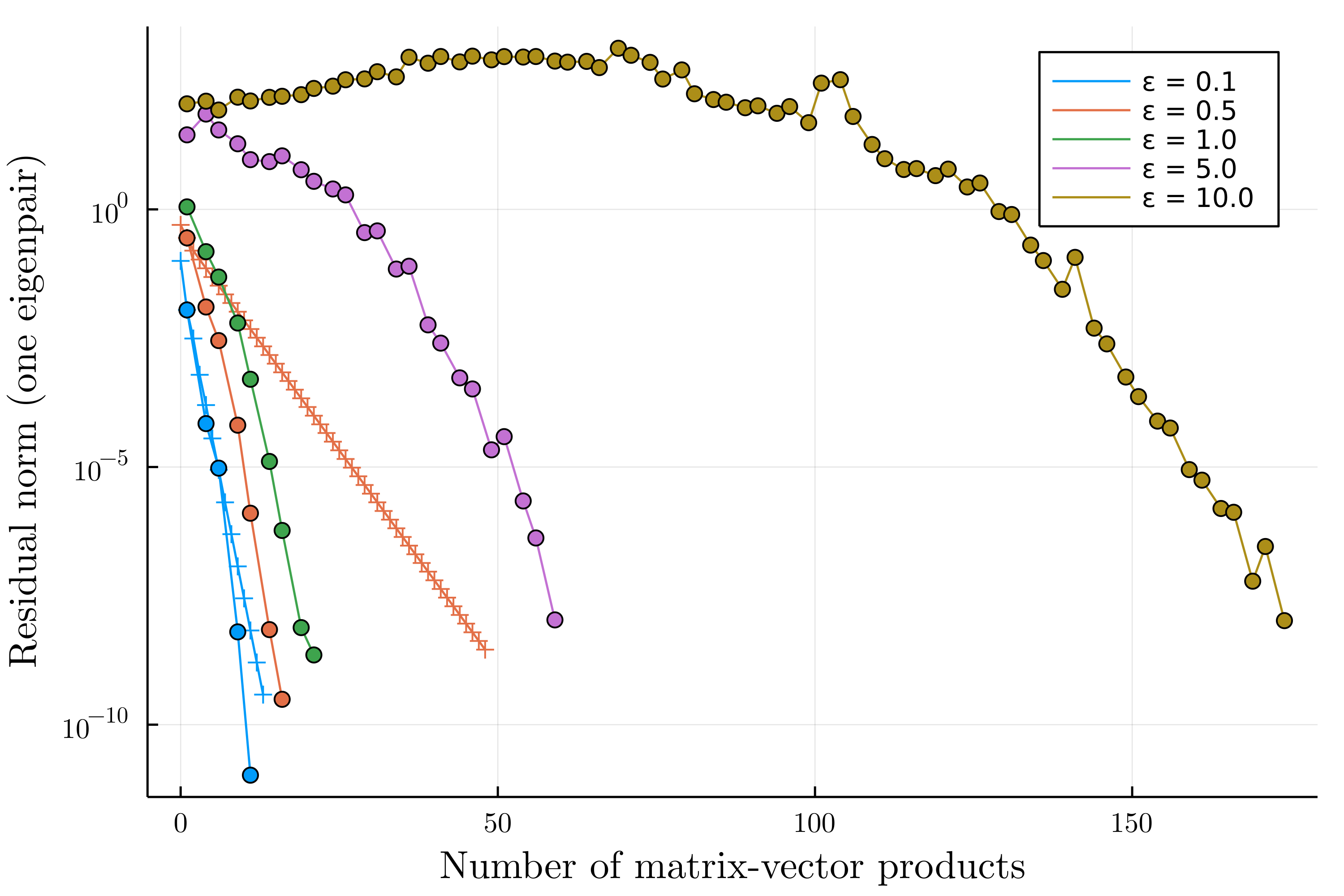}
    \caption{
        IPT and IPT-ACX convergence histories for a $5000\times 5000$ tri-diagonal matrix with off-diagonal element $\varepsilon$. As the latter increases beyond $\simeq 0.5$, IPT (crosses) ceases to converge; IPT-ACX (circles) allows for larger $\varepsilon$ but requires increasingly many matvecs. For $\varepsilon \gtrsim 15$, IPT-ACX also diverges. 
    }
    \label{fail}
\end{figure}

\section{Discussion}

We have presented a new eigenvalue algorithm for near-diagonal matrices, be them symmetric or non-symmetric, dense or sparse. IPT can be applied to obtain a single perturbed eigenvector and often outperforms state-of-the-art preconditioned eigensolvers. IPT can also be applied to the full spectrum problem; in that case IPT benefits from a lower theoretical complexity and lower runtime than dense eigensolvers based on Hessenberg reduction. The largest speed-ups---up to two-orders of magnitude---are obtained for non-symmetric, full spectrum problems. To our knowledge, IPT is the first full-spectrum eigensolver that is able to take initial guesses into account.

\begin{figure}[t]
    \centering
    \includegraphics[width=.9\textwidth]{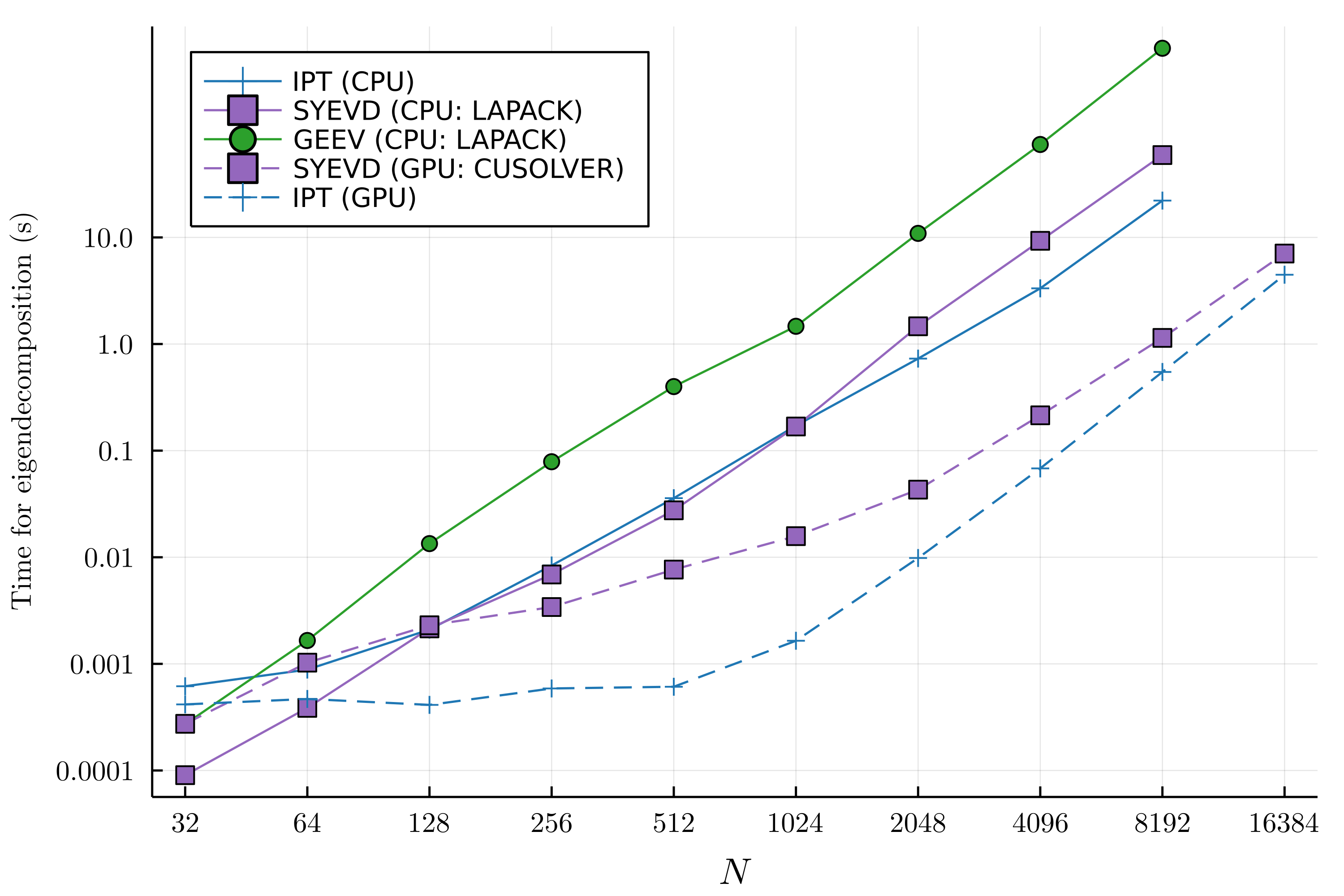}
    \caption{
        Time for the complete diagonalization of random matrices of the form \cref{dense_matrix} with $\varepsilon = 0.01$ and increasing size $n$. Continuous (resp. dashed) lines correspond to computations on the CPU (resp. GPU) using LAPACK (resp. CUSOLVER). The largest speed-up is obtained for non-symmetric matrices. 
    }
    \label{timings}
\end{figure}

Future work should focus on stabilizing our procedure for larger perturbations. For instance, when the perturbation is too large and IPT blows up, we may shrink $\parameter$ it to $\parameter/Q$ for some integer $Q$, diagonalize the matrix with this smaller
perturbation, restart the algorithm using the diagonal matrix thus
obtained as initial condition, and repeat the
operation $Q$ times. This approach is similar to the
homotopy continuation method for finding polynomial
roots \cite{Morgan_2009} and can effectively extend the domain of
convergence of the present iterative scheme. Another idea is to leverage the
projective-geometric structure outlined above, for
instance by using charts on the complex projective space other
than $z_i=1$, which would lead to different maps with
different convergence properties. A third possibility is to use the
freedom in choosing the diagonal matrix $D$ to construct maps
with larger convergence domains, a trick which is known to sometimes
improve the convergence of the RS series \cite{Surj_n_2004}. Finally, we saw that the convergence of IPT can be improved using acceleration methods. It would be interesting to see how far such ideas can go in extending the scope of IPT.

\section*{Acknowledgments}

We thank Rostislav Matveev and Alexander Heaton for useful discussions.

\bibliographystyle{ieeetr}
\bibliography{biblio.bib}

\newpage
\appendix

\section{Additional lemmas for Section~\ref{convergence}}\label{appendix:full_rank}

Here as in the main text we consider an $n \times n$ matrix $M$ with complex entries along with its partition $M = D + \Delta$, where $D$ is diagonal with pairwise distinct diagonal elements. The partition defines a mapping $F\colon \mathbb C^{n\times n} \to \mathbb C^{n\times n}$ as in the main text.

\begin{lemma}
\label{lemma-eigenspace}
Let $M = D + \Delta$ be a matrix with a partition that obeys (\ref{condition-unique}). Then there are no eigenvectors of $M$ from eigenspaces of dimension higher than one as columns of the corresponding unique fixed point $\Z^*$ of $F$ in $B_{\sqrt{2}}(I)$.
\end{lemma}

\begin{proof}
If $\Z^*$ contains an eigenvector from an eigenspace of $M$ with dimension higher than one, then any matrix $\Z$ obtained by a substitution in $\Z^*$ of that eigenvector with a vector from this eigenspace is a fixed point of $F$, too. This means that there is a whole subspace of fixed points of $F$ of dimension higher than zero which includes $\Z^*$. But this contradicts the uniqueness of $\Z^*$ in $B_{\sqrt{2}}(I)$.
\end{proof}

\begin{lemma}
\label{lemma-defective}
If $M$ and its partition obey (\ref{condition-unique}) and $M$ is defective, then the unique fixed point $\Z^*$ of $F$ in $B_{\sqrt{2}}$ does not contain any eigenvector causing the defect.
\end{lemma}

\begin{proof}
By Lemma~(\ref{lemma-eigenspace}), it is enough to assume that $M$ does not have geometrically multiple eigenvalues and all its eigenvectors are isolated projective points, so all fixed points of $F$ are isolated matrices. Let such $M$ be defective. Let $J$ be a Jordan normal form of $M$ and $S$ be the transition matrix that turns $M$ into $J$. Consider a deformation $J_{\mathbf d}$ of $J$ given by $J_{\mathbf d} = J + \diag(\mathbf d)$, where $\mathbf d \in \mathbb C^n$. This induces a deformation of $M$ given by $M_{\mathbf d} = M + S\diag(\mathbf d) S^{-1}$. It is clear that one can find $\mathbf d$ arbitrarily close to 0 (in the standard topology of $\mathbb C^n$) such that all diagonal entries of $J_{\mathbf d}$ are different. As $J_{\mathbf d}$ is upper triangular, its eigenvalues (and thus those of $M_{\mathbf d}$) are equal to its diagonal entries. But then $M_{\mathbf d}$ has a full eigenbasis of isolated eigenvectors. The formerly merged eigenvectors split into distinct ones. Furthermore, the new eigenvectors that split from an old defective eigenvector can be made arbitrarily close to it (as projective points in the standard topology of $\mathbb CP^{n-1}$) by choosing $\mathbf d$ close enough to 0. This is obvious from considering the eigenvectors spawned from a single block of $J_{\mathbf d}$ and their behavior at the limit $\mathbf d \to 0$.

All this implies that there is a deformation $M_{\mathbf d}$ of $M$ such that in its partition $M_{\mathbf d} = D + \Delta_{\mathbf d}$, $\Delta_{\mathbf d} = \Delta + S\diag(\mathbf d)S^{-1}$, $\|\Delta_{\mathbf d} - \Delta\|$ is arbitrarily small and thus $\|G\| \|\Delta_{\mathbf d}\| < 3 - 2\sqrt{2}$ still holds. At the same time, $\Z^*$ splits into several arbitrarily close to it, and thus still contained in $B_{\sqrt{2}}(I)$ for small enough $\mathbf d$, fixed points of the mapping $F_{\mathbf d}$ defined by the partition, with two or more of them being full-rank and multiple associated lower rank fixed points. But this contradicts the uniqueness of the fixed point of $F_{\mathbf d}$ computed for the deformed matrix $M_{\mathbf d}$.
\end{proof}

\begin{theorem}[\cite{Lax2007} Theorem 8, page 130]
\label{theorem-continuous}
Let $M_\parameter$ be a differentiable matrix-valued function of real $\parameter$, $a_\parameter$ an eigenvalue of $M_\parameter$ of multiplicity one. Then we can choose an eigenvector $\mathbf h_\parameter$ of $M_\parameter$ pertaining to the eigenvalue $a_\parameter$ to depend differentiably on $\parameter$.
\end{theorem}

The theorem guarantees that the eigenvector corresponding to an eigenvalue with multiplicity one is a differentiable function of the parameter as a projective space valued function.

\section{Rayleigh-Schr\"odinger perturbation theory}\label{appendix:RSPT}

Here we recall the derivation of the Rayleigh-Schr\"odinger recursion, given \emph{e.g.} in \cite{Roth_2010}. The idea behind the original perturbation theory is the following. Let a parametric family of matrices $M_\parameter = D + \parameter \Delta$ be given, where $\parameter \in \mathbb C$ is the parameter, $M_\parameter \in \mathbb C^{n\times n}$, $M_0 = D$ is diagonal and is called the \emph{unperturbed} matrix, and $\Delta$ is the \emph{perturbation} matrix. The problem is to find eigenvectors of $M_\parameter$ in the form of an asymptotic series in powers of the parameter $\parameter$. We specifically assume a generic case of the unperturbed matrix $D$ having distinct diagonal entries. The perturbation matrix $\Delta$ can be arbitrary.

Let $\eigval_i \in \mathbb C$ and $\mathbf z_i \in \mathbb C^n$ respectively be the $i$-th eigenvalue and the $i$-th eigenvector of $M_\parameter$ ordered such that
\begin{equation}\label{ordering}
\lim_{\parameter \to 0}\eigval_i = \evzero_i, \quad \lim_{\parameter\to 0}\mathbf z_i=\mathbf e_i,
\end{equation}
where $\evzero_i = D_{ii}$, $\mathbf e_i$ are the standard basis vectors of $\mathbb C^n$, and normalized according to $\langle \mathbf e_i,\mathbf z_i\rangle =1$ for all $\parameter$ (a choice known in physics as ``intermediate normalization''), where by $\langle\cdot,\cdot\rangle$ we understand the standard scalar product in $\mathbb C^n$. Such normalization corresponds to singling out eigenvectorst from the affine charts $U_i$ seen as affine subspaces of $\mathbb C^n$ as defined in Section~\ref{fixed-points} of the main text. Therefore, the matrix composed of $\mathbf z_i$ in the chosen order is matrix $\Z$ from the main text. The possibility to order the eigenvectors in such a way that (\ref{ordering}) holds follows from the well known fact that, given the condition on $D$ and the chosen normalization, $\eigval_i$ and $\mathbf z_i$ are holomorphic in $\parameter$ in a certain disk around 0, $\mathbf z_i$ are distinct, and these functions are given by series convergent in that disk \cite{kato1995}:
\begin{equation}\label{power-series}
\eigval_i=\sum_{\ell\geq 0}\parameter^{\ell}\eigval_{i}^{[\ell]}\quad\textrm{and}\quad \mathbf z_i=\sum_{\ell\geq 0}\parameter^{\ell}\mathbf z_{i}^{[\ell]}.
\end{equation}
As a consequence, matrix $\Z$ itself is holomorphic in the same disk and can be represented by a power series in $\parameter$
\begin{equation}\label{A-series}
\Z = \sum_{\ell \geq 0} \parameter^\ell \Z^{[\ell]}.
\end{equation}

Substituting expressions (\ref{power-series}) into the eigenvalue equation $M\mathbf z=\eigval \mathbf z$ and making use of the Cauchy product formula, this yields $D\mathbf z^{[0]}_i=\eigval^{[0]}_i\mathbf z^{[0]}_i$ at zeroth order (hence $\eigval^{[0]}_i=\evzero_i$) and for $\ell\geq 1$
$$(D-\evzero_i I)\mathbf z_{i}^{[\ell]}=\sum_{s=1}^{\ell}\eigval_{i}^{[s]}\mathbf z_{i}^{[\ell-s]}-\Delta \mathbf z_{i}^{[\ell-1]},$$
where $I$ is the unit $n\times n$ matrix.

It is convenient to expand $\mathbf z_{i}^{[\ell]}$ in the basis of the eigenvectors of $D$ as
$\mathbf z_{i}^{[\ell]}=\sum_{j=1}^{n}\Z_{ji}^{[\ell]} \mathbf e_j$ with $\Z_{ij}^{[0]}=\delta_{ij}$ and $\Z_{ii}^{[\ell]}=0$ for $\ell\geq 1$. By construction, $\Z_{ij}^{[\ell]}$ are elements of matrices $\Z^{[\ell]}$ from (\ref{A-series}). This gives for each $\ell\geq 1$ and $1\leq j\leq n$
$$
(\evzero_j-\evzero_i)\Z_{ji}^{[\ell]}=\sum_{s=1}^{\ell}\eigval_{i}^{[s]}\Z_{ji}^{[\ell-s]}-(\Delta \Z^{[\ell-1]})_{ji}.
$$
The equation for the eigenvalues correction is extracted by swapping $i$ with $j$ in this equation and using $\Z_{ii}^{[\ell]}=\delta_{\ell,0}$. This leads to
$\eigval_{i}^{[\ell]}=(\Delta \Z^{[\ell-1]})_{ii}.$
Injecting this back into the equation above we arrive at
\begin{equation}\label{RS-recursion}
\Z_{ij}^{[\ell]}=G_{ij}\Big(\sum_{s=1}^{\ell} (\Delta \Z^{[s-1]})_{jj} \Z_{ij}^{[\ell - s]} - (\Delta \Z^{[\ell-1]})_{ij}\Big).
\end{equation}

\section{Iterative perturbation theory contains the RS
series}\label{appendix:IPT_RSPT}

We prove $\Z^{(k)}  = \Z_\mathrm{RS}^{(k)} + \mathcal O(\parameter^{k+1})$ by induction. Obviously $\Z^{(0)} = \Z_\mathrm{RS}^{(0)} = I$. Suppose that $\Z^{(k-1)} = \Z_\mathrm{RS}^{(k-1)} + \mathcal O(\parameter^k)$ or, more specifically,
$$
\Z^{(k-1)} = \sum_{\ell = 0}^{k-1} \parameter^\ell  \Z^{[\ell]} + \mathcal O(\parameter^k),
$$
where the matrices $\Z^{[\ell]}$ satisfy the recursion (\ref{RS-recursion}). Then from $\Z^{(k)} = F(\Z^{(k-1)})$ we have
$$
\Z^{(k)} = I +\parameter G\circ\left(
\sum_{m=0}^{k-1} \parameter^m \Z^{[m]} \mathcal D \left(\sum_{\ell=0}^{k-1} \parameter^\ell \Delta \Z^{[\ell]} \right) -
\sum_{\ell=0}^{k-1}\parameter^{\ell} \Delta \Z^{[\ell]} \right) + \mathcal O(\parameter^{k+1}).
$$
From this expression it is easy to see that the term of $s$-th order in $\parameter$  in $\Z^{(k)}$ is given by terms with $\ell + m = s - 1$, \emph{viz.} 
$$
\parameter^s G \circ \left(
\left(\sum_{\ell=0}^{s-1} \Z^{[s-1-\ell]} \mathcal D \left(\Delta \Z^{[\ell]}\right)\right)
- \Delta \Z^{[s-1]}
\right).
$$
This term matches exactly the RS correction term $\parameter^s \Z   ^{[s]}$. This concludes the proof.

\section{Explicit examples}\label{appendix:examples}
\label{further-examples}

\subsection{A general approach to find the convergence domain}

Following the notations of the main text, consider the map $\mathbf f_i$ for some fixed $i$. An attracting equilibrium point of the corresponding dynamical system $\mathbf z^{(k)}= \mathbf f_i(\mathbf z^{(k-1)})$ loses its stability when the Jacobian matrix $\partial \mathbf f_i$ of $\mathbf f_i$, $(\partial \mathbf f_i)_{js} \equiv \partial (\mathbf f_i)_j/\partial z_s  = \partial F_{ji}/\partial z_s$, has an eigenvalue (called \emph{multiplier} in this case of a discrete-time dynamical system) with absolute value equal to $1$ at this point. The convergence domain of the dynamical system \cref{dyn} for the whole matrix $\Z$ of the eigenvectors is equal to the intersection of the convergence domains for its individual lines \cref{single_line_dyn}.

Consider the system of $n+1$ polynomial equations of $n+2$ complex variables ($z_j$ for $1\leq j\leq n$, $\parameter$, and $\mu$)

$$
\begin{cases}
\mathbf z = \mathbf f_i(\mathbf z),\\
\det(\partial \mathbf f_i - \mu I) = 0.
\end{cases}
\label{system-one} 
$$
The variable $\mu$ here plays the role of a multiplier of a steady state. Either by successively computing resultants or by constructing a Groebner basis with the correct lexicographical order, one can exclude the variables $z_j$ from this system, which results in a single polynomial of two variables $(\parameter,\mu) \mapsto P(\parameter,\mu)$. This polynomial defines a complex 1-dimensional variety. The projection to the $\parameter$-plane of the real 1-dimensional variety defined by $\{P=0, \vert\mu\vert^2= 1\}$ corresponds to some curve $C$. A more informative way is to represent this curve as a complex function of a real variable $t$ implicitly defined by $P(\parameter,e^{it}) = 0$.

The curve $C$ is the locus where a fixed point of $\mathbf f_i$ have a multiplier on the unit circle. In particular, the fixed point that at $\parameter = 0$ corresponds to $z_i = 1$ and $z_j = 0$, $j \neq i$, loses its stability along a particular subset of this curve. The convergence domain of the iterative perturbation theory is the domain that is bounded by these parts of the curve and that contains $0$. It is possible to show that $C$ is a smooth curve with cusps (return points), which correspond to the values of $\parameter$ such that $M$ has a nontrivial Jordan normal form (is non-diagonalizable), see \cref{appendix:cusps}.

\subsection{The $2\times 2$ example of the main text}

According to the procedure outlined above, the set where IPT loses stability ($C$, the blue curve in  \cref{2d}) is found to be given by the parametric equation $4\parameter^2 + e^{it}(2 - e^{it}) = 0$ (same equation for both eigenvectors). The cusps (return points, the points where $d\parameter/dt = 0$) of this curve are at $\parameter = \pm i/2$. In this particular case, the convergence circle of the RS perturbation theory is completely contained in the convergence domain of the iterative perturbation theory, and their boundaries intersect only at the cusp points. This convergence domain for IPT is directly related to the main cardioid of the classical Mandelbrot set: the set of complex values of the parameter $c$ that lead to a bound trajectory of the classical quadratic (holomorphic) dynamical system $x^{(k+1)} = (x^{(k)})^2 + c$. The main cardioid of the Mandelbort set (the domain of stability of a steady state) is bounded by the curve $4c - e^{it}(2 - e^{it}) = 0$. The boundary of the stability domain of our $2\times2$ example is simply a conformal transform of this cardioid by two complementary branches of the square root function composed with the sign inversion. The origin of this relation becomes obvious after the parameter change $c \mapsto -\parameter^2$ followed by the variable change $x \mapsto \parameter x$. This brings the classical system to the dynamical system of the only nontrivial component of the first line for our $2\times 2$ example: $x^{(k+1)} = \parameter \left((x^{(k)})^2 - 1\right)$. The nontrivial component of the second line follows an equivalent (up to the sign change of the variable) equation: $x^{(k+1)} = \parameter \left(1 - (x^{(k)})^2 \right)$.

This particular $2 \times 2$ example is very special. In fact, any 2-dimensional case is special in the following sense. The iterative approximating sequence for $M = D + \parameter \Delta$ for any $D$ and $\Delta$ takes the form
$$\Z^{(k)} =
\begin{pmatrix}
    1 & f_1^k(0) \\
    f_2^k(0) & 1
\end{pmatrix}
,$$
where $f_j$ are univariate quadratic polynomial functions related by $g_1(x) = x^2 g_2(1/x)$, with $g_j(x) \equiv x - f_j(x)$. The first special feature of this recursion scheme is that it is equivalent to a 1-dimensional quadratic discrete-time dynamical system for each column of $A$. This implies that the only critical point of either $f_j$ ($0$, when the diagonal elements of $\Delta$ are equal) is necessarily attracted by at most unique stable fixed point. The second special feature is fact that both columns have exactly the same convergence domains in the $\parameter$-plane. To see this, suppose that $x_1$ and $x_2$ are the roots of $g_1$. Then it follows that $1/x_1$ and $1/x_2$ are the roots of $g_2$. As $g_2'(1/x) = 2g_1(x)/x - g_1'(x)$, and (like for any quadratic polynomial) $g_1'(x_1) = -g_1'(x_2)$, we also see that $g_1'(x_{1,2}) = g_2'(1/x_{2,1})$, and thus $f_1'(x_{1,2}) = f_2'(1/x_{2,1})$. Therefore, the fixed point of the dynamical systems defined by $x^{(k+1)} = f_1(x^{(k)})$ corresponding to an eigenvector and the fixed point of $x^{(k+1)} = f_2(x^{(k)})$ corresponding to the different eigenvector are stable or unstable simultaneously.

These two properties are not generic when $n>2$. Therefore, we provide another explicit example of a $3 \times 3$ matrix to foster some intuition for more general cases.

\subsection{An additional $3\times 3$ example}

Consider the following parametric $3\times3$ matrix and its partition:
$$
M =
\begin{pmatrix}
0 & 0 & 0 \\
0 & 1 & 0 \\
0 & 0 & 3
\end{pmatrix}
+
\parameter
\begin{pmatrix}
0 & 1 & 2 \\
1 & 0 & 3 \\
2 & 3 & 0
\end{pmatrix}
.
$$
The polynomial $P$ that defines the fixed point degeneration curve $C$ here takes the form for $\mathbf z_1$

\begin{align}
P(\parameter,\mu) &= 63792 \parameter^{7} - 28352 \parameter^{6} \mu - 68040 \parameter^{6} - 29556 \parameter^{5} \mu^{2} + 89352 \parameter^{5} \mu\nonumber\\
&- 13239 \parameter^{5} + 960 \parameter^{4} \mu^{3}
 + 14516 \parameter^{4} \mu^{2} - 39164 \parameter^{4} \mu + 12116 \parameter^{4}\nonumber\\ 
&+ 5616 \parameter^{3} \mu^{4} - 26658 \parameter^{3} \mu^{3} + 29988 \parameter^{3} \mu^{2} - 546 \parameter^{3} \mu - 2448 \parameter^{3}\nonumber\\ 
&+ 468 \parameter^{2} \mu^{5} - 3720 \parameter^{2} \mu^{4} + 12820 \parameter^{2} \mu^{3} - 17648 \parameter^{2} \mu^{2} + 7584 \parameter^{2} \mu\nonumber\\ 
&- 1296 \parameter^{2} - 243 \parameter \mu^{6} + 1404 \parameter \mu^{5} - 2619 \parameter \mu^{4} + 1350 \parameter \mu^{3} + 432 \parameter \mu^{2}\nonumber\\ 
&+ 108 \mu^{6} - 792 \mu^{5} + 1980 \mu^{4} - 1872 \mu^{3} + 432 \mu^{2},
\label{poly-n1} 
\end{align}
for $\mathbf z_2$

\begin{align}
P(\parameter,\mu) &= 113408 \parameter^{7} - 63792 \parameter^{6} \mu - 120960 \parameter^{6} + 7416 \parameter^{5} \mu^{2} + 53208 \parameter^{5} \mu\nonumber\\  
&+ 36424 \parameter^{5} + 6525 \parameter^{4} \mu^{
3} - 11034 \parameter^{4} \mu^{2} - 13824 \parameter^{4} \mu - 5664 \parameter^{4}\nonumber\\ 
&- 3156 \parameter^{3} \mu^{4} - 2472 \parameter^{3} \mu^{3} + 10332 \parameter^{3} \mu^{2} + 3696 \parameter^{3} \mu + 3088 \parameter^{3}\nonumber\\  &- 72 \parameter^{2} \mu^{5} + 1800 \parameter^{2} \mu^{4} - 1800 \parameter^{2} \mu^{3} - 2088 \parameter^{2} \mu^{2} - 1296 \parameter
^{2} \mu\nonumber\\ 
&- 576 \parameter^{2} + 128 \parameter \mu^{6} - 24 \parameter \mu^{5} - 736 \parameter \mu^{4} - 120 \parameter \mu^{3} + 1328 \parameter \mu^{2}\nonumber\\ 
&- 72 \mu^{6} + 108 \mu^{5} + 360 \mu^{4} - 432 \mu^{3} - 288 \mu^{2},
\label{poly-n2} 
\end{align}
and for $\mathbf z_3$

\begin{align}
P(\parameter,\mu) &= 35440 \parameter^{7} - 42528 \parameter^{6} \mu - 37800 \parameter^{6} - 32360 \parameter^{5} \mu^{2} + 110080 \parameter^{5} \mu\nonumber\\ 
&- 23295 \parameter^{5} + 29112 \parameter^{4} \mu^
{3} - 4800 \parameter^{4} \mu^{2} - 88614 \parameter^{4} \mu + 51024 \parameter^{4}\nonumber\\ 
&+ 14640 \parameter^{3} \mu^{4} - 78760 \parameter^{3} \mu^{3} + 116760 \parameter^{3} \mu^{2} - 52920 \parameter^{3} \mu + 5400 \parameter^{3}\nonumber\\ 
&- 2376 \parameter^{2} \mu^{5} - 10152 \parameter^{2} \mu^{4} + 70296 \parameter^{2} \mu^{3} - 101496 \parameter^{2} \mu^{2} + 41904 \parameter^{2} \mu\nonumber\\ 
&- 864 \parameter^{2} - 1080 \parameter \mu^{6} + 7920 \parameter \mu^{5} - 19620 \parameter \mu^{4} + 19440 \parameter \mu^{3} - 6480 \parameter \mu^{2}\nonumber\nonumber\\ 
&+ 1296 \mu^{6} - 7992 \mu^{5} + 17712 \mu^{4} - 16416 \mu^{3} + 5184 \mu^{2}.
\label{poly-n3} 
\end{align}
As we can see, the dynamical systems for all three columns of $\Z$ ($\mathbf z_1$, $\mathbf z_2$, and $\mathbf z_3$) have different domains of convergence in the $\parameter$ plane. The corresponding curves are depicted in \cref{473965}, \cref{988969}, and \cref{824432}, respectively.

There are differences also in curves for individuals columns with those for the $2\times 2$ case. Note that a curve $C$ does not contain enough information to find the convergence domain itself. The domains on \cref{473965}--\cref{824432} were found empirically. Of course, they are bound by some parts of $C$ and include the point $\parameter = 0$. The reason for some parts of $C$ not forming the boundary of the domain is that its different parts correspond to different eigenvectors. In other words, they belong to different branches of a multivalued eigenvector function of $\parameter$, the cusps being the branching points.

Consider as a particular example the case $\mathbf z_2$. The curve $C$ intersects itself at $\parameter \approx -0.49$. Above the real axis about this point, one of the two intersecting branches of the curve form the convergence boundary. Below the real axis, the other one takes its place. This indicates that the two branches correspond to different multipliers of the same fixed point. When $\Im \parameter > 0$, one of them crosses the unitary circle at the boundary of the convergence domain;  when $\Im \parameter < 0$, the other one does. At $\parameter \approx -0.49$, both of them cross the unitary circle simultaneously. This situation corresponds, thus, to a Neimark-Sacker bifurcation (the discrete time analog of the Andronov-Hopf bifurcation). Both branches are in fact parts of the same continuous curve that passes through the point $\parameter \approx 0.56$. Around this point, the curve poses no problem to the convergence of the dynamical system. The reason for this is that an excursion around cusps (branching points of eigenvectors) permutes some eigenvectors. As a consequence, the curve at $\parameter \approx -0.49$ corresponds to the loss of stability of the eigenvector that is a continuation of the unique stable eigenvector at $\parameter = 0$ by the path $[0,-0.49]$. At the same time, the same curve at $\parameter \approx 0.56$ indicates a unitary by absolute value multiplier of an eigenvector that is not a continuation of the initial one by the path $[0,0.56]$.

Not all features of the curves depicted this $3\times3$ case are generic either. The particular symmetry with respect to the complex conjugation of the curve and of its cusps is not generic for general complex matrices $D$ and $\Delta$, but it is a generic feature of matrices with real components. In this particular case, due to this symmetry, the only possible bifurcations for real values of $\parameter$ are the flip bifurcations (a multiplier equals to $-1$, typically followed by the cycle doubling cascade), the Neimark-Sacker bifurcation (two multipliers assume complex conjugate values $e^{\pm it}$ for some $t$), and, if the matrices are not symmetric, the fold bifurcation (a multiplier is equal to $1$). With symmetric real matrices, the fold bifurcation is not encountered because the cusps cannot be real but instead form complex conjugated pairs. These features are the consequence of the behavior of $\det(D + \parameter \Delta - x I)$ with respect to complex conjugation and from the fact that symmetric real matrices cannot have nontrivial Jordan forms.

Likewise, Hermitian matrices result in complex conjugate nonreal cusp pairs, but the curve itself is not necessarily symmetric. As a result, there are many more ways for a steady state to lose its stability, from which the fold bifurcation is, however, excluded. Generic bifurcations at $\parameter \in \mathbb R$ here consist in a multiplier getting a value $e^{it}$ for some $t \neq 0$. The situation for general complex matrices lacks any symmetry at all. Here steady states lose their stability by a multiplier crossing the unit circle with any value of $t$, and thus the fold bifurcation, although possible, is not generic. It is generic for one-parameter (in addition to $\parameter$) families of matrices.

As already noted, for the holomorphic dynamics of  any $2\times2$ case the unique critical point is guaranteed to be attracted by the unique attracting periodic orbit, if the latter exists. This, in turn, guarantees that for any $\Delta$ with zero diagonal the iteration of $F$ starting from the identity matrix converges to the needed solution provided that $\parameter$ is in the convergence domain. This is not true anymore for $n > 2$, starting already from the fact that there are no discrete critical points in larger dimensions. The problem of finding a good initial condition becomes non-trivial. As can be seen on ~\cref{988969}, the particular $3\times3$ case encounters this problem for the second column ($\mathbf z_2$). The naive iteration with $\Z^{(0)} = I$ does not converge to the existing attracting fixed point of the dynamical system near some boundaries of its convergence domain. Our current understanding of this phenomenon is the crossing of the initial point by the attraction basin boundary (in the $\mathbf z$-space). This boundary is generally fractal. Perhaps this explains the eroded appearance of the empirical convergence domain of the autonomous iteration.

To somewhat mitigate this complication, we applied a nonautonomous iteration scheme in the form, omitting details, $\mathbf z^{(k+1)} = \mathbf f_2(\mathbf z^{(k)},\parameter (1 - \alpha^k))$ with $\mathbf z^{(0)} = (0,1,0)^T$, where $\alpha$ is some positive number $\alpha < 1$, so that $\lim_{k \to \infty} \parameter (1 - \alpha^k) = \parameter$, and we explicitly indicated the dependence of $\mathbf f_2(\mathbf z,\parameter)$ on $\parameter$. The idea of this \emph{ad hoc} approach is the continuation of the steady state in the extended $(\mathbf z,\parameter)$-phase space from values of $\parameter$ that put $\mathbf z^{(0)}$ inside the convergence domain of that steady state. Doing so, we managed to empirically recover the theoretical convergence domain (see ~\cref{988969}).

Finally, we would like to point out an interesting generic occurrence of a unitary multiplier without the fold bifurcation. For $\mathbf z_1$, this situation takes place at $\parameter \approx 0.45$, for $\mathbf z_2$ at $\parameter \approx 0.56$, and for $\mathbf z_3$ at $\parameter \approx 1.2$. All three points are on the real axis, as  is expected from the symmetry considerations above. There is no cusp at these points and no fold bifurcations (no merging of eigenvectors), as it should be for symmetric real matrices. Instead, another multiplier of the same fixed point goes to infinity at the same value of $\parameter$ (the point becomes super-repelling). As a result, the theorem of the reduction to the central manifold is not applicable.
\begin{figure}[H]
\begin{center}
\includegraphics[width=0.70\columnwidth]{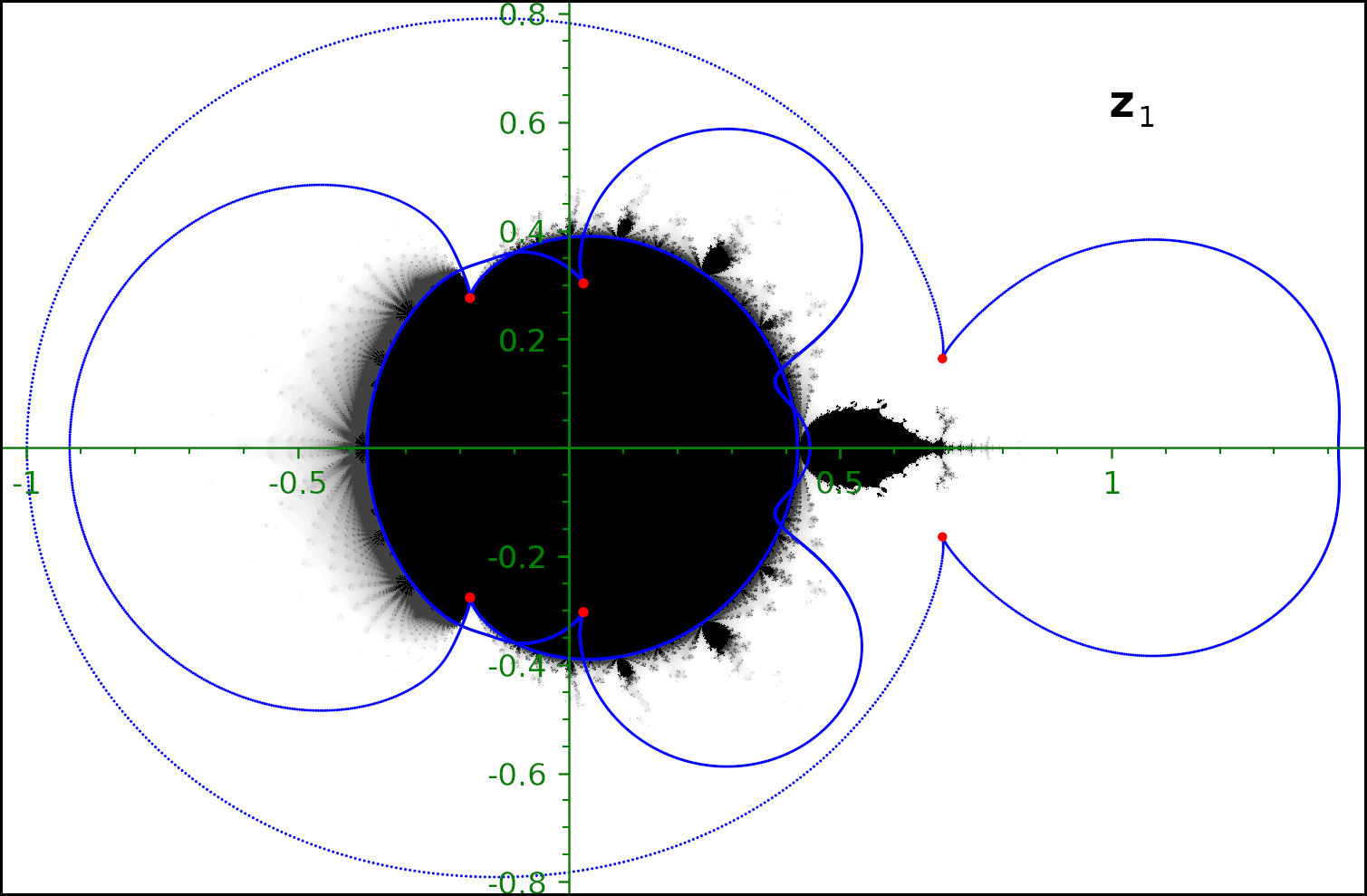}
\caption{{The convergence domain on the~\(\parameter\)-plane for the first
column of~\(\Z\)~(the first eigenvector \(\mathbf z_1\))
for the~\(3\times3\) example. The Mandelbrot-like set (domain where orbits remain bounded) of the iterative scheme is shown in black and grey. The empirical convergence domain is shown in black. Its largest component corresponds to the stability of a steady state (the applicability domain of the iterative method). Small components correspond to stability of various periodic orbits. Various shades of grey show the values of $\parameter$ that lead to divergence to infinity (the darker the slower the divergence). In red are the values of $\parameter$ where the matrix is non-diagonalizable.
{\label{473965}}
}}
\end{center}
\end{figure}
\begin{figure}[H]
\begin{center}
\includegraphics[width=0.70\columnwidth]{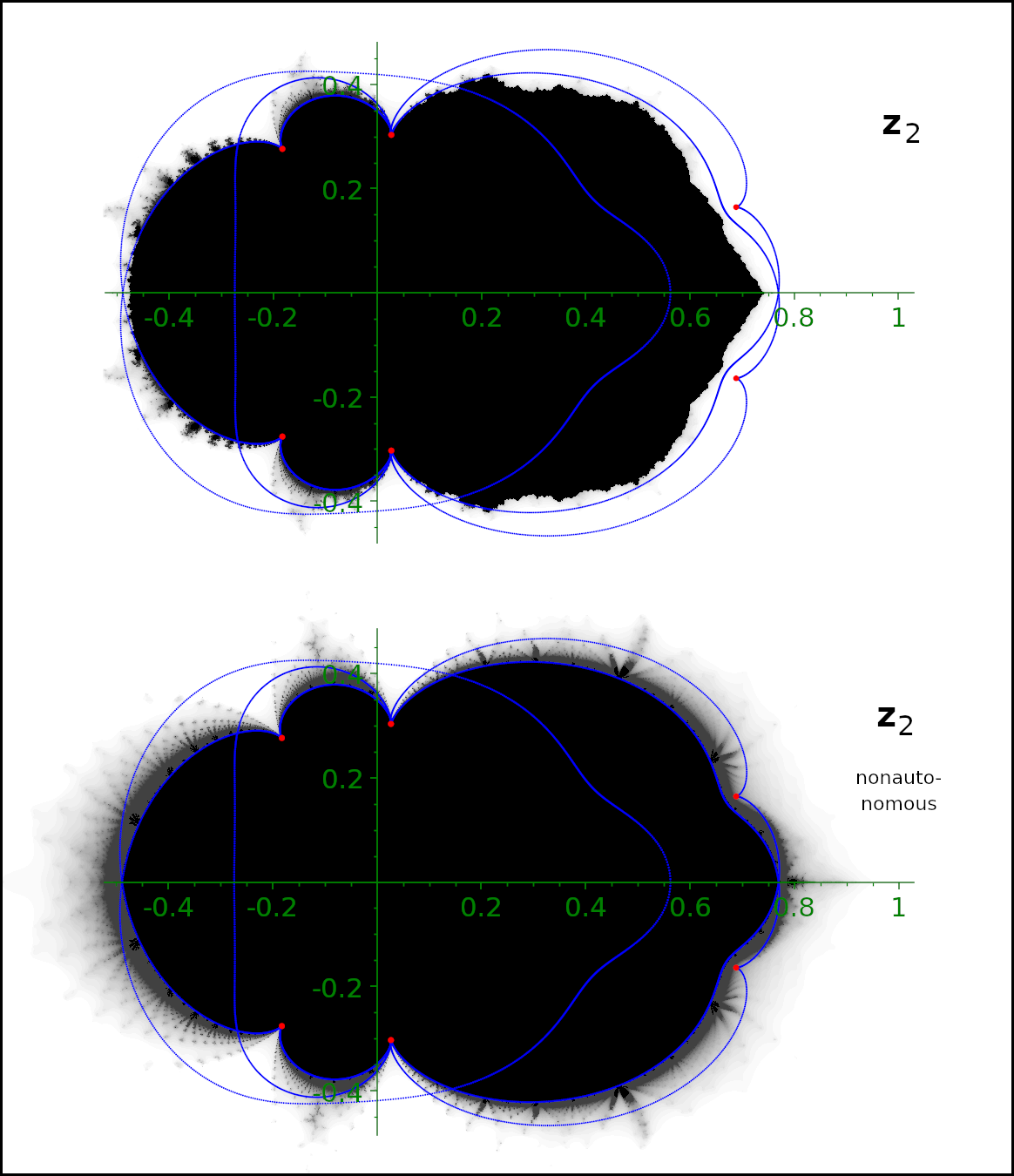}
\caption{{{Same as ~{\cref{473965}} for the second
eigenvector~\(\mathbf z_2\) (the second column of \(\Z\)).
Small components correspond to stability of various periodic orbits.
Various shades of grey show the values of~\(\parameter\) that lead
to divergence to infinity (the darker the slower the divergence).
{\label{988969}}%
}%
}}
\end{center}
\end{figure}
\begin{figure}[H]
\begin{center}
\includegraphics[width=0.70\columnwidth]{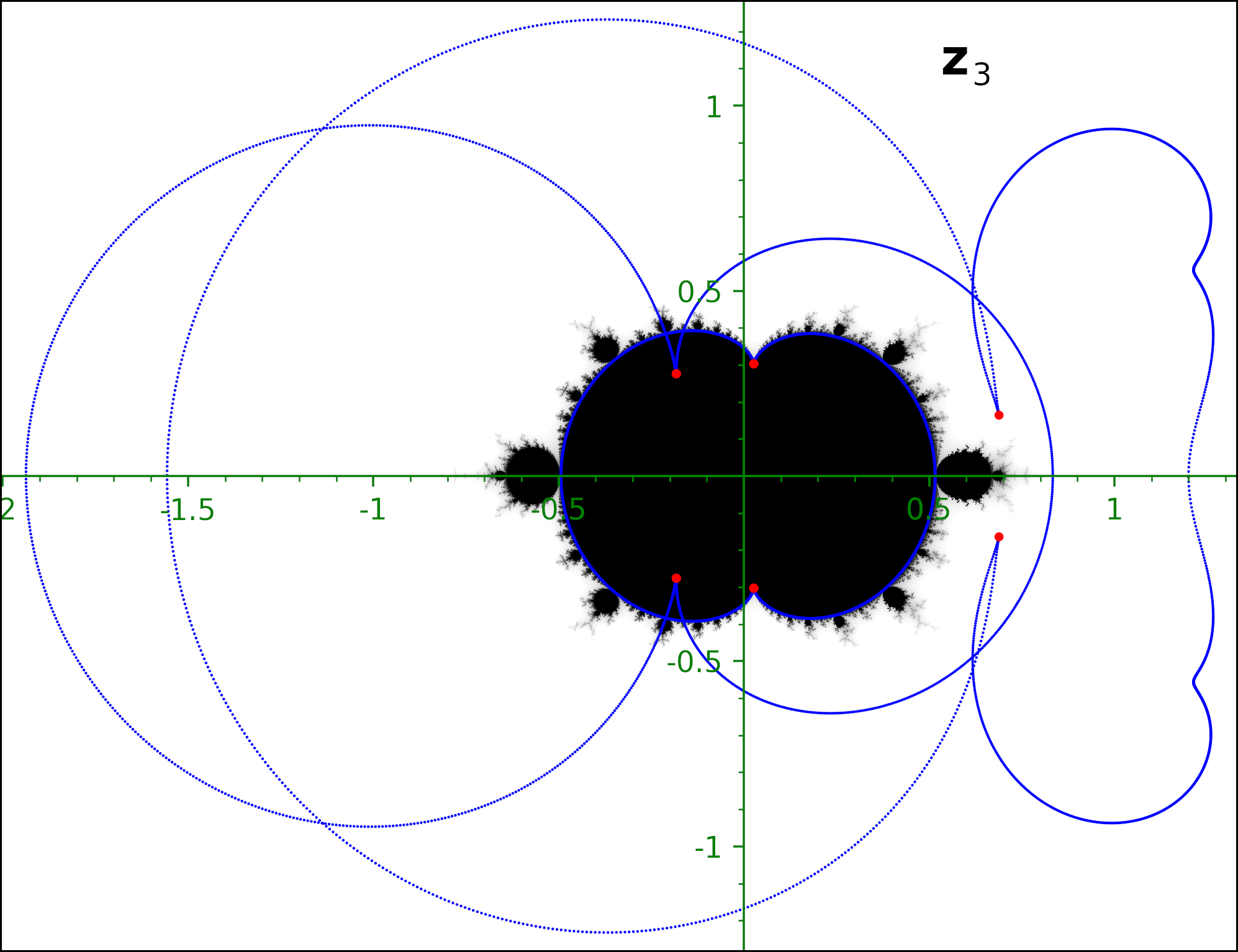}
\caption{{{Same as ~{\cref{473965}} for the third
eigenvector~\(\mathbf z_3\) (the third column of~\(\Z\)) .
{\label{824432}}%
}%
}}
\end{center}
\end{figure}

\section{Exceptional points and cusps of $C$}\label{appendix:cusps}

In a large part of this section we drop the previously adopted notation conventions for vectors. We use instead the more appropriate notation style from differential geometry.

We will prove that if $M$ is defective at some value of $\parameter$, then $d\parameter/dt = 0$ for the curve $C$ at that point (as above, $\mu$ is a multiplier of the dynamical perturbation theory and $C$ is defined by $\mu = e^{it}$ and is locally considered as a curve $t \mapsto \parameter(t)$).

Fix $i$ and consider the dynamical system for the $i$-th column only. Let $\mathbf x = (x_1,\ldots,x_{n-1})$ be a tuple of affine coordinates in the affine chart $U_i$. Note that the coordinates are rearranged in comparison to $z_j$: $x_1 = z_1$, \ldots, $x_{i-1} = z_{i-1}$, $x_i = z_{i+1}$, \ldots, $x_{n-1} = z_n$. We consider the representation of the corresponding dynamical system in $U_i$ too. Let it be given by the tuple of functions $\mathbf h = (h_1,\ldots,h_{n-1})$ that correspond to functions in $\mathbf f_j$ with $(\mathbf f_i)_i$ omitted ($(\mathbf f_i)_i$ is trivially 1 in $U_i$). The same rearrangement is implied for $h_j$ as for $x_j$. Thus, the dynamical system is defined by $\mathbf x^{(k+1)} = \mathbf h(\mathbf x^{(k)})$ and the stationary states are defined by the system of equations $\mathbf x = \mathbf h(\mathbf x)$.

Consider $\mathbb C^{n+1}$ with coordinates $(x_1,\ldots,x_{n-1},\parameter,\nu)$, where $\nu \equiv \mu - 1$, as a complex analytic manifold with these coordinates as global holomorphic coordinates on it. Let us define polynomial functions
$$\mathcal F_j\colon \mathbb C^{n+1} \to \mathbb C, (\mathbf x,\parameter,\nu) \mapsto h_j(\mathbf x,\parameter) - x_j,$$
and $\mathbf {\mathcal F} = (\mathcal F_1,\ldots,\mathcal F_{n-1})$. Let us denote $J \equiv \dfrac{\D\mathcal F}{\D \mathbf x} \equiv \dfrac{\D(\mathcal F_1,\ldots,\mathcal F_{n-1})}{\D(x_1,\ldots,x_{n-1})}$ the Jacobian matrix of $\mathcal F$ with respect to variables $x_j$. Let $I$ be the unitary $(n-1)\times(n-1)$ matrix.

Consider the complex 1-dimensional variety $\mathcal C \subset \mathbb C^{n+1}$ defined by the polynomial system
$$
\begin{cases}
\mathcal F = 0,\\
\det(J - \nu I) = 0.
\end{cases}
$$
Curve $C$ is the projection to the $\parameter$-plane of the real 1-dimensional variety $\tilde C = \mathcal C \cap \{|\nu +1|^2 = 1\}$ in $\mathbb C^{N+1}$ considered as $\mathbb R^{2(N+1)}$.

First, note that if $M$ is defective at some value of $\parameter$, then the system of equations $\mathcal F = 0$ (with this values of $\parameter$ fixed and considered for unknowns $x \in \mathbb C^{n-1}$) has a root $x$ of multiplicity greater than 1. This means that the hyperplanes $\{\mathcal F_j =0\}$ are not in general position at the intersection that corresponds to this root, which implies $\det \D\mathcal F/\D \mathbf x = 0$. Therefore, the point $p \in \mathbb C^{n+1}$ with the same $x$ and $\parameter$ and with $\nu = 0$ belongs to $\mathcal C$ and represents this non-diagonalizability of $M$.

Let $d$ denote the exterior derivative and $\wedge$ denote the exterior product on the complex of holomorphic exterior forms $\Omega^\bullet(\mathbb C^{n+1})$. Alternatively, one may treat it in purely axiomatic way as the K\"ahler differential on the algebra of holomorphic functions on $\mathbb C^{n+1}$ with the appropriate factorization in the end. Let $p \in \mathcal C$ be a point of geometric degeneration of $M$ with $\nu = 0$ as above.

\begin{theorem}
Assume that the following nondegeneration condition holds: $d_p \det J \wedge \bigwedge\limits_j d_p \mathcal F_j \neq 0$. Then $\mathcal C$ can be locally parametrized by $\nu$ around $p$ and, with this parametrization, $d\parameter/d\nu = 0$ at $p$.
\end{theorem}

\begin{proof}
Let $T_p \mathbb C^{n+1}$ be the holomorphic tangent space to $\mathbb C^{n+1}$ at $p$, that is the tangent space spanned by the holomorphic vector fields $\D_j \equiv \D/\D x_j$, $\D_\parameter \equiv \D/\D \parameter$, $\D_\nu \equiv \D/\D \nu$ at $p$. Let us denote $\iota_u \sigma_p$ the contraction of a holomorphic form $\sigma$ ($\sigma_p \in \bigwedge^\bullet_p \mathbb C^{n+1}$) with a tangent vector $u \in T_p \mathbb C^{n+1}$ at point $p$.

Let us denote $\omega_p \equiv d_p \det (J - \nu I) \wedge \bigwedge\limits_j d_p \mathcal F_j$ and $\varpi_p \equiv d_p \det J \wedge \bigwedge\limits_j d_p \mathcal F_j$. By the premise, $\varpi_p \neq 0$, which also implies $\omega_p \neq 0$. Indeed, the free term (with respect to $\nu$) of the polynomial $\det(J - \nu I)$ is equal to $\det J$, and thus

\begin{equation}
d_p \det(J - \nu I) = \D_\nu \det(J - \nu I)|_p\, d_p \nu + d_p \det J,
\label{eq-nu} 
\end{equation}
where the two terms are linearly independent because $\D_\nu \det J = 0$. Therefore, as non of $\mathcal F_j$ depends on $\nu$, $\omega_p$ differs from $\varpi_p$ by an addition of a linearly independent term.

Let $v \in T_p \mathbb C^{n+1}$ be a nonzero vector with coordinates $(v^i,v^\parameter,v^\nu)$ tangent to $\mathcal C$. It means that $v\det (J - \nu I) = v\mathcal F_j = 0$, where by $vf$ we denote the action of a vector $v$ on a function $f$. This, in turn, implies $\iota_v \omega_p = 0$.

As all $\mathcal F_j$ depend only on $x$ and $\parameter$, we have $d_p \parameter \wedge \bigwedge\limits_j d_p \mathcal F_j = \det J|_p\, d_p \parameter \wedge \bigwedge\limits_j d_p x_j = 0$, and thus $d_p \parameter \wedge \omega_p = 0$. Therefore, we have
$$
\iota_v (d_p \parameter \wedge \omega_p) = -d_p \parameter \wedge \iota_v \omega_p + v^\parameter \omega_p = v^\parameter \omega_p = 0.
$$
This implies $v^\parameter = 0$.

On the other hand, by \cref{eq-nu} we have
$$d_p \nu \wedge d_p \det(J - \nu I) = d_p \nu \wedge d_p \det J,$$
and thus $d_p \nu \wedge \omega_p = d_p \nu \wedge \varpi_p \neq 0$. But $d \nu \wedge \omega \in \Omega^{n+1}(\mathbb C^{n+1})$, and therefore, for any holomorphic tangent vector $u \in T_p \mathbb C^{n+1}$, $u \neq 0$ is equivalent to $\iota_u (d_p \nu \wedge \omega_p) \neq 0$. This results in
$$
\iota_v (d_p \nu \wedge \omega_p) = -d_p \nu \wedge \iota_v \omega_p + v^\nu \omega_p = v^\nu \omega_p \neq 0,
$$
and thus $v^\nu \neq 0$.

By the holomorphic implicit function theorem, $\mathcal C$ can be holomorphically parametrized by $\nu$ in a neighborhood of $p$. Together with $v^\parameter = 0$ it implies that we have $d\parameter/d\nu|_p = 0$ on $\mathcal C$.
\end{proof}

Now, consider a smooth real curve parametrized by a real parameter $t$ on the complex $\nu$-plane that without degeneracy passes through 0. This curve is lifted to $\mathcal C$ and the resulting smooth real curve is parametrized by $t$. From $d\parameter/d\nu = 0$ we conclude that $d\parameter/dt = 0$ at $p$ too. In our case we have the curve $\mu(t) = e^{it}$ or $\nu(t) = e^{it} - 1$, which passes through $\mu = 1$ at $t = 0$. $\tilde C$ is projected without degeneration to $\mathbb C \times \mathbb R \ni (\parameter,t)$ locally near $p$ and then to $\mathbb C \ni \parameter$ with degeneration at the projection of $p$ given by $d\parameter/dt = 0$.

\end{document}